\documentclass[preprint,1p,times,sort&compress]{elsarticle}
\usepackage{amsmath,amsfonts,mathrsfs}
\usepackage{amssymb}
\usepackage{color}
\usepackage{mathrsfs}
\usepackage{mathptmx}
\usepackage{verbatim}
\usepackage{epsfig}
\usepackage{CJK}
\usepackage{bm}
\usepackage{amsfonts}
\usepackage{amsthm}
\usepackage[pagewise]{lineno}

\newtheorem{theorem}{Theorem}[section]
\newtheorem{lemma}{Lemma}[section]

\newtheorem{corollary}{Corollary}[section]

\newdefinition{rem}{Remark}[section]

\numberwithin{equation}{section}

\def\mi{{\bf i}}
\def\e{\varepsilon}

\begin{document}
\begin{frontmatter}
\title{Reducibility for wave equations of finitely smooth potential with periodic boundary conditions}
\author[SDUJ,SDUW]{Jing Li}
\ead{xlijing@sdu.edu.cn}
\author[FDU]{Yingte Sun\corref{cor}}
\ead{sunyt15@fudan.edu.cn}
\author[SDUJ,SDUW]{Bing Xie}
\cortext[cor]{Corresponding author.} \ead{xiebing@sdu.edu.cn}
\address[SDUJ]{School of Mathematics , Shandong University, Jinan 250100, P.R. China}
\address[SDUW]{School of Mathematics and Statistics, Shandong University, Weihai 264209, P.R. China}
\address[FDU]{School of Mathematics, Fudan University, Shanghai 200433,  P.R. China}

\journal{}

\begin{abstract}
In the present paper, the reducibility is derived for the wave equations with finitely smooth and
time-quasi-periodic potential subjects to periodic boundary conditions.
More exactly, the linear wave equation $u_{tt}-u_{xx}+Mu+\varepsilon (V_0(\omega t)u_{xx}+V(\omega t, x)u)=0,\;x\in \mathbb{R}/2\pi \mathbb{Z}$
can be reduced to a linear Hamiltonian system of a constant coefficient operator which is of pure imaginary point
spectrum set, where $V$ is finitely smooth in
$(t, x)$, quasi-periodic in time $t$ with Diophantine frequency $\omega\in \mathbb{R}^{n},$ and $V_0$ is finitely smooth and
 quasi-periodic in time $t$ with Diophantine frequency $\omega\in \mathbb{R}^{n},$
Moreover, it is proved that the corresponding wave operator possesses the property of pure point
spectra and zero Lyapunov exponent.
\end{abstract}

\end{frontmatter}

\section{Introduction}

In the present paper, we investigate the reducibility of
 \begin{equation}\label{1*}
u_{tt}-u_{xx}+Mu+\varepsilon (V_0(\omega \, t)u_{xx}+V(\omega \, t,x)u)=0,\;\;x\in\mathbb{R}/2\pi \mathbb{Z}.
\end{equation}
To that end, we need the following conditions:

\ \

{\bf Assumption A.} {\it Assume $M>0$ is a constant, and
$V_0,V_1$ are $C^N$-smooth and  quasi-periodic in time $t$ with frequency $\omega\in \mathbb R^n$: that is,
there are hull functions $\mathcal{V}_0(\theta)\in {C}^{N}(\mathbb{T}^{n}, \mathbb{R})$, $\mathcal{V}(\theta,x)\in {C}^{N}(\mathbb{T}^{n}\times [0, 2\pi], \mathbb{R})$
such that
$$ V_0(\omega\, t)=\mathcal{V}_0(\theta)|_{\theta=\omega t}, \  V(\omega\, t,x)=\left.\mathcal{V}(\theta, x)\right|_{\theta=\omega t},
\;\;\mathbb{T}^{n}=\mathbb{R}^{n}/2\pi\mathbb{Z}^{n},$$
where $N>200\, n$.}

\ \

{\bf Assumption B.} {\it
Assume $\omega\in [1, 2]^{n}\subset \mathbb{R}^{n}$ satisfies Diophantine conditions:

\begin{equation}\label{eq3}
|\langle k, \omega \rangle |\geq \frac{\gamma}{|k|^{n+1}},\;\;k\in \mathbb{Z}^{n}\setminus \{0\},
\end{equation}
where $\gamma$ is a constant and $0<\gamma\ll 1$.}

\ \

We recall the reducibility problem for a time dependent linear system
 \begin{equation}\label{1.1}
 \dot x=A(t)x,\; x\in \mathbb{R}^n,
 \end{equation}
 where $A(t)$ is an $n \times n$ real or complex value matrix. If $A(t)$ is time $T$- periodic and continuous, it follows from Floquet theory that there exists a
 continuous time $T$- periodic coordinate change
 \begin{equation}\label{1.2}
 x=P(t)y
 \end{equation}
 such that \eqref{1.1} is changed into a constant system
 \begin{equation}\label{1.3}
 \dot y=By,
 \end{equation}
 where $B$ is an $n \times n$ complex value matrix independent of time $t$. However, there usually does not exist the change \eqref{1.2} such that
  \eqref{1.1} is reduced to \eqref{1.3} when $A(t)$ is time quasi-periodic. See \cite{Eliasson98}. Let us consider a special case:
  $A(t)=\Lambda+\varepsilon Q(t),$ where $\Lambda$ is a constant, $Q(t)$ is time quasi-periodic and $\varepsilon$ is small. The well known KAM (Kolmogorov-Arnold-Moser)
  theory can be applied to this case. See \cite{Bogoljubov1969,Eliasson92,Jorba92,Li-Zhu}, for example. In recent decades, there have been many literatures  dealing with the
  reducibility of time quasi-periodic, infinite dimensional linear systems via KAM technique. One model is the time-quasi-periodic Schr\"{o}dinger operator
  \begin{equation}\label{1.4}
  \mi\,\dot u= (H_{0}+\e W(\omega \, t, x, -\mi\,\nabla u))u,\;\; x\in\mathbb R^d \;\;\mbox{or}\;\;x\in\mathbb T^d=\mathbb R^d/2\pi \mathbb{Z}^{d},
  \end{equation}
  where $H_0=-\triangle +V(x)$ or an abstract self-adjoint (unbounded) operator while the perturbation $W$ is quasi-periodic in time $t$ and it may or may not depend on $x$ or/and $\nabla$. See \cite{Bambusi17arxiv, Bambusi17CMP, Combescure1987Ann, Duclos-Stovicek96, Duclos-Stovicek02, Eliasson98, Wang17nonlinearity}, and the references therein.

Another model is the time-quasi-periodic wave operator or linear wave equation
\begin{equation}\label{1.5}
u_{tt}=(-\triangle+\e V(\phi_0+\omega t,x;\omega))u.
\end{equation}
Up to now, the reducibility of \eqref{1.5} has not been explicitly dealt with. Note that a reducibility procedure has been included in classical KAM for the existence of lower-dimensional invariant tori for infinite dimensional Hamiltonian partial differential equations. It can be implicitly derived from the classical KAM \cite{Chierchia-You,Kuk1,Poschel1996,Way} that
      \eqref{1.5} with $d=1$ and subject to Dirichlet boundary condition or periodic boundary condition can be
  reduced to a constant coefficient equation for ``most
  \footnote{Here the word ``most" means that for a given set $\Pi\subset \mathbb{R}^{n}$ with Lebesgue measure equals to 1,
  there exists a subset $\Pi_{\varepsilon}\subset \Pi$ with measure $\Pi\setminus \Pi_{\varepsilon} \rightarrow 0$ as $\varepsilon \rightarrow 0$ such that
  for ``any $\omega\in\Pi_{\varepsilon}$".}" frequency $\omega,$ provided that $V$ is analytic. For $d=1$ and \eqref{1.5} with a finitely smooth potential $V$ and subject to Dirichlet boundary condition, it has been recently proved that \eqref{1.5} can still
  be reduced to a constant system for ``most" frequency $\omega.$ See \cite{Li}.

\ \

In this paper, we will prove the following reducibility theorem:
\begin{theorem}\label{thm1.1} With Assumptions {\bf A, B}, for any given $0<\gamma\ll 1$,
 there exists an $\epsilon^*$ with $ \;0<\varepsilon^{*}=\varepsilon^{*} (n, \gamma)\ll \gamma,$
and exists a subset $\Pi\subset [1, 2]^{n}$ with
$$\mbox{Measure}\, \Pi\geq 1-O (\gamma^{1/3})$$
such that for any $0<\varepsilon<\varepsilon^{*}$ and for any $\omega \in \Pi,$ there is a quasi-periodic  symplectic change
such that
\begin{equation}\label{eq1}
 u_{tt}-u_{xx}+Mu+\varepsilon(V_0(\omega \, t)u_{xx}+ V(\omega \, t,x)u)=0,\;\;x\in\mathbb{R}/2\pi \mathbb{Z}
\end{equation}
is reduced to a linear Hamiltonian system
\begin{equation}\label{yuan1.7}
\left\{
  \begin{array}{ll}
    \dot{\tilde{q}}=({\Lambda}+\varepsilon \tilde{Q})\tilde{p}, \\
    \dot{\tilde{p}}=-({\Lambda}+\varepsilon \tilde{Q})\tilde{q},
  \end{array}
\right.
\end{equation}
where
$\Lambda=diag \left({\Lambda_{j}}: j=0, 1, 2, \cdots\right),$
${\Lambda_{0}}=\rho\sqrt{M},\;\Lambda_{j}=\rho\sqrt{j^{2}+M}E_{22},$
$\rho$ is a constant close to $1$, $E_{22}$ is a $2\times 2$ unit matrix,
 and
$\tilde{Q}=diag (\tilde{Q}_{i}: \;i=0, 1, 2, \cdots)$ is independent of time with
$\tilde{Q}_{0}\in\mathbb{R}$,
$\tilde{Q}_{i}
$ being a real $2\times 2$ matrix, and $|\tilde{Q}_{i}|\leq C/i,\;i=1, 2, \cdots.$ Here $|\cdot|$ denotes the sup-norm for real matrices.
\end{theorem}
The
more exact statement of Theorem 1.1 can be found in Theorem 2.1 in Section 2. From Theorem
1.1, the following two corollaries can be obtained.
\begin{corollary}\label{col1}
With Assumptions {\bf A, B}, for $\omega\in \Pi$ and $0<\varepsilon<\varepsilon^{*},$ the wave operator
\[\mathcal L \, u(t,x)=(\partial_t^2-\partial_x^2+M+\varepsilon (V_0(\omega\, t)\partial_x^2+V(\omega\, t,x))u(t,x),\;\; x\in\mathbb{R}/2\pi \mathbb{Z} \]
is of pure point spectrum property and of zero Lyapunov exponent.
\end{corollary}
\begin{corollary}\label{col2}
With Assumptions {\bf A, B}, for any $\omega\in \Pi$ and $0<\varepsilon<\varepsilon^{*},$ there exists a unique solution
$u(t, x)$ with initial values $(u(0, x), u_{t}(0, x))=(u_{0}(x), v_{0}(x))\in\mathcal{H}^{N}\times \mathcal{H}^{N},$ which is
almost-periodic in time and
$$\frac{1}{C}(\|u_{0}\|_{\mathcal{H}^{N}}+\|v_{0}\|_{\mathcal{H}^{N}})\leq \|u(t)\|_{\mathcal{H}^{N}}+\|u_{t}(t)\|_{\mathcal{H}^{N}}\leq C(\|u_{0}\|_{\mathcal{H}^{N}}+\|v_{0}\|_{\mathcal{H}^{N}}),$$ where $C>0$ is a constant, $\mathcal{H}^{N}=\mathcal{H}^{N}(\mathbb{T}^{n})$
is the usual Sobolev space.
\end{corollary}
\begin{rem}\label{rem1.1}
Since  $V_0(\omega t)\partial_{xx}$ appears in \eqref{1*}, the perturbation is unbounded one. This kind of unbounded perturbation, which is of the highest unboundedness, can come from the linearization of some quasi-linear perturbation. For quasi-linear Kdv equations and quasi-linear Schr\"{o}dinger equations, there has been a progress about KAM theory 
\cite{Baldi,Baldi2016, Baldi16mKdv, Giuliani, Feola2016, Feola2015, Riccardo Montalto1, Berti2017, Baldi17, Roberto}. It is still an open problem whether or not there exists KAM theory for quasi-linear wave equations. In the present paper, the potential $V_{0}(\omega t)$ in \eqref{1*} does not depend on the space variable $x$. We find that the methods of Baldi-Berti-Montalto \cite{Baldi,Riccardo Montalto1,Riccardo Montalto} and Roberto-Michela \cite{Roberto} is still valid for the $V_{0}(\omega t)$ in \eqref{1*}.

\end{rem}
\begin{rem}
Here we would like to compare the results of Theorem \ref{thm1.1} with some existent results. As mentioned before, without $V_0(\omega t)$, when $d=1$ and the potential $V$ is
analytic, the reducibility of \eqref{1.5} can be implicitly derived from the classical KAM theorems. However, there are some differences between the analytic potential $V$
 and the finitely smooth one, not to mention the existence of $V_0$. In this paper, by several times elegant variable and symplectic changes, the wave equation \eqref{1*} can be written as a linear Hamiltonian system with Hamiltonian
 $$H=\langle\widetilde{{\Lambda}} z,\overline{z}\rangle+\varepsilon \left[\langle \widetilde{R}^{zz}(\theta)z, z\rangle
+\langle \widetilde{R}^{z\overline{z}}(\theta)z, \overline{z}\rangle+\langle \widetilde{R}^{\overline{z}\overline{z}}(\theta)\overline{z}, \overline{z}\rangle\right].$$
  See \eqref{eq21} for more details. The basic task is to search a series of symplectic coordinate changes to eliminate the perturbations $\widetilde{R}^{zz}(\theta),$  $\widetilde{R}^{z\overline{z}}(\theta)$  and $\widetilde{R}^{\overline{z}\overline{z}}(\theta)$ except the averages of the diagonal of $\widetilde{R}^{z\overline{z}}(\theta).$ To this end, the symplectic coordinate changes are the time-1 map of the flow for the Hamiltonian $\varepsilon F$ where $F$ is of the form
 $$F=\langle F^{zz}(\theta)z, z\rangle+\langle F^{z \overline{z}}(\theta)z, \overline{z}\rangle+\langle F^{\overline{z}\,\overline{z}}(\theta)\overline{z}, \overline{z}\rangle.$$
 \begin{itemize}
   \item  When the potential $V(\theta)\;(\theta=\omega t)$ is analytic in some strip domain $|\text{Im} \theta|\leq s_{\nu}^{*},$ (where $\nu$ is the KAM iteration step), the perturbations
 $\widetilde{R}^{zz}(\theta),$ $\widetilde{R}^{z\overline{z}}(\theta)$ and $\widetilde{R}^{\overline{z}\overline{z}}(\theta)$ are also analytic in $|\text{Im} \theta|\leq s^{*}_{\nu}.$ An important fact in this analytic
 case is that $s^{*}_{\nu}$'s have a uniform non-zero below bound:
 $$s_{\nu}^{*}\geq \frac{s_{0}}{2},\;s_{0}>0,\;\;\mbox{for all}\;\;\nu=1, 2, \cdots .$$
   \item When the potential $V(\theta)$ is finitely smooth of order $N,$ by using Jackson-Moser-Zehnder approximate lemma, we can still make sure that $\widetilde{R}^{zz}(\theta),$ $\widetilde{R}^{z\overline{z}}(\theta)$ and $\widetilde{R}^{\overline{z}\overline{z}}(\theta)$ are analytic in $|\text{Im} \theta|\leq s_{\nu}$ at the $\nu-$ th KAM step. However, the strip width $s_{\nu}$'s have no non-zero below bound. Actually,  $s_{\nu}$ goes to zero very rapidly:
       $$s_{\nu}=\varepsilon_{\nu+1}^{1/N},\;\;\varepsilon_{\nu}=\varepsilon^{(4/3)^{\nu}},\;\;\nu=1,2, \cdots.$$
 \end{itemize}
\begin{itemize}
    \item For the analytic case, we can prove the Hamiltonian $\varepsilon F=O(\varepsilon_{\nu})$ at the $\nu-$th KAM
    step, because $s_{\nu}^{*}\geq \frac{s_{0}}{2}.$ It follows immediately that the new perturbation is $\{\varepsilon F, \varepsilon R\}=O(\varepsilon_{\nu}^{2})=O(\varepsilon_{\nu+1}).$
    \item For the finitely smooth case, the situation is much more complicated. At this case, we find
    $\varepsilon F=O(\varepsilon_{\nu}^{1-\frac{2(3n+4)}{N}})$  at the $\nu-$th KAM step. Thus, for the finitely smooth potential $V\in C^N$, the new perturbation is
    $\{\varepsilon F, \varepsilon R\}=O(\varepsilon_{\nu}^{2-\frac{2(3n+4)}{N}}).$ In order to guarantee the quadratic convergence of the KAM iterations, that is, $O(\varepsilon_{\nu}^{2-\frac{2(3n+4)}{N}})=O(\varepsilon_{\nu}^{4/3})=O(\varepsilon_{\nu+1}),$ it is necessary to assume the smoothness order $N>>1.$ It is enough to assume $N>200 n.$ Clearly, this is not sharp. In this paper, We do not pursuit the lowest smoothness for the potential $V$.
\end{itemize}
\end{rem}

\begin{rem}
The reducibility of \eqref{1*} with finitely smooth potential $V$ subject to Dirichlet boundary condition has been derived in a recent paper \cite{Li}.
However, the results on the reducibility between Dirichlet boundary condition and periodic boundary condition are different. For Dirichlet boundary condition,
the eigenvalues $\lambda_{j}\;(j=1, 2, \cdots)$ are simple. Thus, we can reduce the Hamiltonian
$$H=\langle \widetilde{{\Lambda}} z, \overline{z}\rangle+
\varepsilon (\langle \widetilde{R}^{zz}(\theta)z, z\rangle+\langle \widetilde{R}^{z\overline{z}}(\theta)z, \overline{z}\rangle+\langle \widetilde{R}^{\overline{z}\overline{z}}(\theta)\overline{z}, \overline{z}\rangle)$$
to $$H_{\infty}=\langle \widetilde{\widetilde{\Lambda}}z, z\rangle,$$
where $\widetilde{\widetilde{\Lambda}}=diag(\widetilde{\lambda}_{j}: j=1, 2, \cdots)$ and
$\widetilde{\lambda}_{j}=\sqrt{j^{2}+M}+\xi_{j}.$ Moreover, \eqref{1*} can be reduced to
$$u_{tt}-u_{xx}+M_{\xi}u=0,$$ where $M_{\xi}$ is a Fourier multiplier.
However, for periodic boundary condition, the eigenvalues $\lambda_{j}\;(j=0, 1, \cdots)$ are double:
$$\lambda_{0}^{\sharp}=1,\;\;\lambda_{j}^{\sharp}=2,\;\;j=0, 1, \cdots.$$
In this case, the Hamiltonian $H$ can be reduced to
$$H_{\infty}=\langle ({\Lambda}+\varepsilon \widetilde{Q})u, \overline{u}\rangle,$$
where ${\Lambda} $ and $\widetilde{Q}$ are matrices defined as \eqref{yuan1.7}, $u$ is a vector defined as \eqref{2*} . Although we can still get some dynamical
behaviour from this reducibility, \eqref{1*} can not be reduced to a linear wave
equation with a Fourier multiplier as in Dirichlet boundary condition.
\end{rem}
\begin{rem}
Since $\lambda_{j}^{\sharp}=2,$ the homological equations are no longer scalar. For example, in order to eliminate the term
$\langle R^{u\overline{u}}(\theta)u, \overline{u}\rangle$ (see \eqref{eq211}-\eqref{eq22.3} for more details), the homological equations have the form:
\begin{equation}\label{1.9}
\omega\cdot\partial_{\theta} F-\mi (\Lambda F-F\Lambda)=R,
\end{equation}
where $F=F(\theta)$ is the unknown matrix of order 2, $\Lambda$ is a $2\times 2$ constant matrix, $R=R(\theta)$ is known matrix of order 2. It is more complicated to
find the solution of this matrix equation \eqref{1.9} than that of scalar
homological equations. In this case, the delicate small divisor problem becomes
one dealing with the inverse of the matrix
\begin{equation}\label{1.10}
A:=-\langle k, \omega\rangle(1\otimes 1)+1 \otimes \Lambda-\Lambda\otimes 1
\end{equation}
(see \eqref{eq7.41} for more details). A usual method dealing with
\eqref{1.10} is to investigate $\partial_{\omega}^{4}\det A.$
See \cite{Bogoljubov1969} and \cite{Chierchia-You}, for example. In
the present paper, we use the variation principle of eigenvalues to
deal with the inverse $A^{-1}$. The advantage of the variation principle
of eigenvalues is that the method dealing with scalar small divisor
problem \cite{Poschel1996} can be recovered.
\end{rem}

\begin{rem} \label{rem2}
In \cite{Massimiliano-Berti}, it is proved that there is a quasi-periodic solution for any $d$-dimensional nonlinear wave
equation with a quasi-periodic in time nonlinearity,
\[u_{tt}-\Delta u-V(x)\, u=\varepsilon f(\omega \, t, x, u),\;\; x\in \mathbb T^d,\]
where the multiplicative potential $V$ is in $C^q(\mathbb T^d;\, \mathbb R),\; \omega\in\mathbb R^n$ is a non-resonant frequency vector and $f\in C^q(\mathbb T^n\times \mathbb T^d\times\mathbb R;\, \mathbb R)$.
Because of the application of multi-scale-analysis, it is not clear whether the obtained quasi-periodic solution is linear stable and has zero Lyapunov exponent. As a corollary  of Theorem \ref{thm1.1},  we can prove that the quasi-periodic solution by \cite{Massimiliano-Berti} is linear stable and has zero Lyapunov exponent, when $d=1$.
\end{rem}
\begin{rem} \label{rem3}
When $d>1,$ it is a well-known open problem that \eqref{1.5} subject to Dirichlet or periodic boundary condition is reduced to a linear Hamiltonian system with a constant coefficient
linear operator. See the series of talks by L.H. Eliasson \cite{1,2,3}.  Also see a recent paper \cite{Riccardo Montalto} where the perturbation is a finite rank operator.
\end{rem}

This paper is organized as follows. In Section 2, we redescribe
Theorem 1.1 as Theorem 2.1. In Section 3-10, to prove the main results of the paper, some
preliminary work and many lemmas will be given. The proof of Theorem 2.1 is in the last
section.


\section{Passing to Fourier coefficients}

Consider the differential equation:
\begin{equation}\label{eq4}
\mathcal{L}u=u_{tt}-u_{xx}+Mu+\varepsilon(V_0(\theta)u_{xx}+V(\omega t, x)u)=0
\end{equation}
subject to the boundary condition

\begin{equation}\label{eq5}
u(t,x)=u(t,x+2k\pi),\;\;k\in \mathbb{Z}.
\end{equation}
It is well-known that the Sturm-Liouville problem

\begin{equation*}\label{eq6}
-y''+My=\lambda y,\;\;\;'=\frac{d}{dx},\;\;x\in \mathbb{R}/2\pi\mathbb{Z}
\end{equation*}
has the eigenvalues and eigenfunctions, respectively,

\begin{equation*}\label{eq8}
\lambda_{k}=k^{2}+M,\;\;\;\;k\in \mathbb{Z},
\end{equation*}
\begin{equation*}\label{eq9}
\phi_{k}(x)=e ^{\mi\, kx},\;\;\;\;k\in \mathbb{Z}.
\end{equation*}
Set $-\partial_{xx}+M$ as $D$, the wave equation can be seen as
\begin{equation}
u_{tt}=-Du+\varepsilon V_0(\omega t)Du- \varepsilon V_1(\omega t,x)u,
\end{equation}
where $V_1(\omega t,x)=V(\omega t,x)+M V_0(\omega t)$. Let $u_t=v$, we have
\begin{equation}
v_t=-(1-\varepsilon V_0(\omega t))Du- \varepsilon V_1(\omega t,x)u.
\end{equation}
$\mathbf{Step} \ \mathbf{1:}$

Rescale
\begin{equation*}\label{eq091}
\quad \left\{%
\begin{array}{ll}
u=\beta(\theta)|D|^{-\frac{1}{4}}q,\\
\\
v=(\beta(\theta))^{-1}|D|^{\frac{1}{4}}p.\\
\end{array}%
\right.
\end{equation*}
Then
\begin{eqnarray*}\label{eq091}
\left\{%
\begin{array}{ll}
q_t=\frac{1}{\beta^2(\theta)}|D|^{\frac{1}{2}}p-\frac{\omega \cdot \partial_{\theta}\beta(\theta)}{\beta(\theta)}q,\\
\\
p_t=-(1-\varepsilon V_0(\omega t)) \beta^2(\theta)|D|^{\frac{1}{2}}q+\frac{\omega \cdot \partial_{\theta}\beta(\theta)}{\beta(\theta)}p-\varepsilon|D|^{-\frac{1}{4}}\beta^2(\theta) V_1(\omega t,x)|D|^{-\frac{1}{4}}q.\\
\end{array}%
\right.
\end{eqnarray*}
Choose a suitable $\beta(\theta)$, such that $\beta(\theta)=(1-\varepsilon V_0(\theta))^{-\frac{1}{4}}$. Then
$$\frac{1}{\beta^2(\theta)}=(1-\varepsilon V_0(\omega t))\beta^2(\theta)\triangleq a_0(\theta).$$
Also, set $\frac{\omega\cdot\partial_{\theta}\beta(\theta)}{\beta(\theta)}=\varepsilon a_1(\theta)$, $\beta^2(\theta)V_1(\theta,x)=\widetilde{V_1}(\theta,x)$, we have
\begin{equation*}\label{eq092}
 \left\{%
\begin{array}{ll}
q_t=a_0(\theta)|D|^{\frac{1}{2}}p-\varepsilon a_1(\theta)q,\\
\\
p_t=-a_0(\theta)|D|^{\frac{1}{2}}q+\varepsilon a_1(\theta)p-\varepsilon|D|^{-\frac{1}{4}}\widetilde{V_1}(\theta,x)|D|^{\frac{1}{4}}q.\\
\end{array}%
\right.
\end{equation*}
Clearly, we can see $a_0,\tilde{V}_1 \in {C}^{N}(\mathbb{T}^{n}\times [0, 2\pi], \mathbb{R})$ and $a_1 \in {C}^{N-1}(\mathbb{T}^{n}\times [0, 2\pi], \mathbb{R})$.

$\mathbf{Step} \ \mathbf{2:}$

Now we consider the complex variable
$$z=\frac{q-\mathbf{i}p}{\sqrt{2}}, \quad \bar{z}=\frac{q+\mathbf{i}p}{\sqrt{2}}.  $$
Then, we have
\begin{equation}\label{eq093}
\quad \left\{%
\begin{array}{cl}
\omega\cdot \partial_\theta z=\mi a_0(\theta)|D|^{\frac{1}{2}}z- \varepsilon a_1(\theta)\bar{z}+ \varepsilon \mathbf{i}|D|^{-\frac{1}{4}}\frac{\widetilde{V_1}(\theta,x)}{2}|D|^{-\frac{1}{4}}(z+\bar{z}),\\
\\
\omega\cdot \partial_\theta \bar{z}=-\mi a_0(\theta)|D|^{\frac{1}{2}}\bar{z}-\varepsilon a_1(\theta)z-\varepsilon \mathbf{i}|D|^{-\frac{1}{4}}\frac{\widetilde{V_1}(\theta,x)}{2}|D|^{-\frac{1}{4}}(z+\bar{z}).\\
\end{array}%
\right.
\end{equation}
$\mathbf{Step} \ \mathbf{3:}$

Now we introduce a time variable change, a diffeomorphism of the torus $\mathbb{T}^{n}$ of the form
\begin{equation}
\vartheta=\theta+\omega a(\theta), \quad  \theta=\vartheta+\omega \tilde{a}(\vartheta).
\end{equation}
For any function $h(\theta,x)$ and $\tilde{h}(\vartheta,x)$, we introduce  operators $A$ and  $A^{-1}$, where
\begin{equation}
\begin{split}
h(\theta,x)&=(A^{-1} h)(\vartheta,x)=[h](\vartheta,x)=h(\vartheta+\omega \tilde{a}(\vartheta),x),\\
\tilde{h}(\vartheta,x)&=(A \tilde{h})(\theta,x)=\tilde{h}(\theta+\omega {a}(\theta),x).
\end{split}
\end{equation}
Our aim is to rewrite the equation \eqref{eq093} in the new time variable $\vartheta$. Thus, we can
set
\begin{equation}
\begin{split}
 z(\theta,t)&=z(\vartheta+\omega \tilde{a}(\vartheta),x)=[z](\vartheta,x),\\
 a_i(\theta)&=a_i(\vartheta+\omega \tilde{a}(\vartheta))=[ a_i](\vartheta), \ i=0,1, \\
 \widetilde{V_1}(\theta,x)&=\widetilde{V_1}(\vartheta+\omega \tilde{a}(\vartheta),x)=[\widetilde{V_1}](\vartheta,x),\\
 1+\omega\partial_{\theta}a(\theta)&=1+\omega\partial_{\theta}a(\vartheta+\omega \tilde{a}(\vartheta))=[1+\omega\partial_{\theta}a](\vartheta),
\end{split}
\end{equation}
\begin{equation*}\label{eq094}
\mathcal{T}:\quad \left\{%
\begin{array}{cl}
\omega\cdot \partial_\vartheta [z]=\mathbf{i}\frac{[a_0]}{[1+\omega\partial_{\theta}a]}|D|^{\frac{1}{2}}z-\varepsilon \frac{[a_1]}{[1+\omega\partial_{\theta}a]}\bar{z}+\varepsilon \mathbf{i}|D|^{-\frac{1}{4}}\frac{[\widetilde{V_1}(\theta,x)]}{2[1+\omega\partial_{\theta}a]}|D|^{-\frac{1}{4}}([z]+[\bar{z}]),\\
\\
\omega\cdot \partial_\vartheta [\bar{z}]=-\mathbf{i}\frac{[a_0]}{[1+\omega\partial_{\theta}a]}|D|^{\frac{1}{2}}\bar{z}-\varepsilon \frac{[a_1]}{[1+\omega\partial_{\theta}a]}z-\varepsilon \mathbf{i}|D|^{-\frac{1}{4}}\frac{[\widetilde{V_1}(\theta,x)]}{2[1+\omega\partial_{\theta}a]}|D|^{-\frac{1}{4}}([z]+[\bar{z}]).\\
\end{array}%
\right.
\end{equation*}
We want to choose a function $a$ so that $[a_0]$ is proportional to $[1+\omega\partial_{\theta}a]$. Thus, it is enough to solve the equation
\begin{equation}\label{2.9*}
\rho(1+\omega\partial_{\theta}a(\theta))=a_0(\theta), \quad \rho \in \mathbb{R}.
\end{equation}
Integrating on $\mathbb{T}^n$ we fix the value of $\rho$ as
\begin{equation}
\rho=\frac{1}{(2\pi)^n}\int_{\mathbb{T}^n}a_0(\theta)d\theta.
\end{equation}
By \eqref{2.9*}, we get
\begin{equation}
a(\theta)=(\omega\cdot \partial_{\theta})^{-1}[\frac{a_0}{\rho}-1](\theta).
\end{equation}
For notational simplicity, rename $\vartheta$ ,$[z], [\bar{z}],\frac{[a_1]}{[1+\omega\partial_{\theta}a]},\frac{[\widetilde{V_0}(\theta,x)]}{[1+\omega\partial_{\theta}a]}$ as  $\theta$ $z,\bar{z}, b_0, V$. Then, we have
\begin{equation*}\label{eq094}
\mathcal{T}:\quad \left\{%
\begin{array}{cl}
 z_t=\mathbf{i}\rho|D|^{\frac{1}{2}}z-\varepsilon b_0\bar{z}+\varepsilon\mathbf{i}|D|^{-\frac{1}{4}}\frac{V}{2}|D|^{-\frac{1}{4}}(z+\bar{z}),\\
\\
\bar{z}_t=-\mathbf{i}\rho|D|^{\frac{1}{2}}\bar{z}-\varepsilon b_0z-\varepsilon\mathbf{i}|D|^{-\frac{1}{4}}\frac{V}{2}|D|^{-\frac{1}{4}}(z+\bar{z}).\\
\end{array}%
\right.
\end{equation*}

By Sobolev embedding theorem and inverse function theorem, we see $a  \in {C}^{N-2n-2}(\mathbb{T}^{n}\times [0, 2\pi])$ and $\tilde{a} \in \in {C}^{N-2n-2}(\mathbb{T}^{n}\times [0, 2\pi])$. Thus, we can get $b_0,V \in {C}^{N-2n-3}(\mathbb{T}^{n}\times [0, 2\pi])$. In the following section, we renamed $N-2n-3$ as $N$ for notational simplicity.

Make the ansatz
\begin{equation*}\label{eq10}
z(t, x)=\mathcal{S} (z_k)=\sum_{k\in \mathbb{Z}} z_{k}(t)\phi_{k}(x), \quad \bar{z}(t, x)=\mathcal{S} (\bar{z}_k)=\sum_{k\in \mathbb{Z}} \bar{z}_{k}(t)\phi_{k}(x)
\end{equation*}
and
\begin{equation*}\label{eq11}
V(\omega t, x)=\sum_{k\in \mathbb{Z}} v_{k}(\omega t)\phi_{k}(x).
\end{equation*}

Then \eqref{eq4} can be transformed as
\begin{equation}\label{eq16}
\begin{split}
\frac{dz_{k}}{dt}= \mathbf{i} \rho\sqrt{\lambda_{k}}z_{k}-\varepsilon b_0\bar{z}_k+ \mi \varepsilon \sum_{l\in\mathbb{Z}}\sum_{j\in\mathbb{Z}}\frac{c_{jlk}v_{j}}{2\sqrt[4]{\lambda_{k}\lambda_{l}}}(z_{l}+\bar{z}_{l}),\\
\frac{d\bar{z}_{k}}{dt}=-\mathbf{i}\rho \sqrt{\lambda_{k}}\bar{z}_{k}-\varepsilon b_0z_k-\mi \varepsilon \sum_{l\in\mathbb{Z}}\sum_{j\in\mathbb{Z}}\frac{c_{jlk}v_{j}}{2\sqrt[4]{\lambda_{k}\lambda_{l}}}(z_{l}+\bar{z}_{l}),\\
\end{split}
\end{equation}
where
\begin{equation}\label{eq2.11-}
c_{jlk}
=\int_{0}^{2\pi}e^{\mi\, (j+l-k)x}\, dx=
\left\{%
\begin{array}{ll}
 \;\;0,\;j+l-k\neq 0,\\
2\pi,\; j+l-k=0.\\
\end{array}%
\right.
\end{equation}

Endowed a symplectic transformation with $-\mi d z\wedge d \overline{z}$. Thus \eqref{eq16} is changed into
\begin{equation}\label{eq19}
\left\{%
\begin{array}{cl}
&\dot{z}_{k}=\mi\, \frac{\partial H}{\partial \overline{z}_{k}},\;\;k\in\mathbb{Z},\\
\\
&\dot{\overline{z}}_{k}= - \mi\, \frac{\partial H}{\partial z_{k}},\;\;k\in\mathbb{Z},\\
\end{array}%
\right.
\end{equation}
where
\begin{equation}\label{eq20}
H(z,\overline{z})=\sum_{k\in\mathbb{Z}}\rho \sqrt{\lambda_{k}}z_{k}
\overline{z}_{k}+\varepsilon\mi \sum_{k\in\mathbb{Z}}b_0(\frac{\bar{z}^2_k-z^2_k}{2})+\varepsilon \sum_{k\in\mathbb{Z}}\sum_{l\in\mathbb{Z}}\sum_{j\in\mathbb{Z}}
c_{jlk}\frac{v_{j}(\theta)}{2\sqrt[4]{\lambda_{k}\lambda_{l}}}
\left(z_{l}+\overline{z}_{l}\right)\left(z_{k}+\overline{z}_{k}\right).
\end{equation}
For two sequences $x=(x_{j}\in \mathbb{C},\;j\in \mathbb{Z})$,\;
$y=(y_{j}\in \mathbb{C},\;j\in \mathbb{Z}),$
define
$$\langle x, y\rangle=\sum_{j\in \mathbb{Z}}x_{j}y_{j}.$$
Then we can rewrite \eqref{eq20} as follows:
\begin{equation}\label{eq21}
H(z,\overline{z})=\langle\rho \widetilde{\Lambda} z,\overline{z}\rangle+\varepsilon\mi \frac{b_0}{2}\Big( \langle \overline{z},\overline{z}\rangle-\langle z,z\rangle\Big)+ \varepsilon \left[\langle \widetilde{R}^{zz}(\theta)z, z\rangle
+\langle \widetilde{R}^{z\overline{z}}(\theta)z, \overline{z}\rangle+\langle \widetilde{R}^{\overline{z}\overline{z}}(\theta)\overline{z}, \overline{z}\rangle\right],
\end{equation}
where
\begin{equation*}
\widetilde{\Lambda}=diag \left(\sqrt{\lambda_{j}}: j\in \mathbb{Z}\right),\;\;\theta=\omega t,
\end{equation*}
\begin{equation}\label{eq22.11}
\widetilde{R}^{zz}(\theta)=\left(\widetilde{R}^{zz}_{kl}(\theta): k, l\in \mathbb{Z}\right),
\;\; \widetilde{R}^{zz}_{kl}(\theta)=\frac{1}{2}\sum_{j\in \mathbb{Z}}\frac{c_{jlk}v_{j}(\theta)}{\sqrt[4]{\lambda_{k}}\sqrt[4]{\lambda_{l}}},
\end{equation}
\begin{equation}\label{eq22.21} \widetilde{R}^{z\overline{z}}(\theta)=\left(\widetilde{R}^{z\overline{z}}_{kl}(\theta): k, l\in \mathbb{Z}\right),
\;\; \widetilde{R}^{z\overline{z}}_{kl}(\theta)=\sum_{j\in \mathbb{Z}}\frac{c_{jlk}v_{j}(\theta)}{\sqrt[4]{\lambda_{k}}\sqrt[4]{\lambda_{l}}},
\end{equation}
\begin{equation}\label{eq22.31} \widetilde{R}^{\overline{z}\,\overline{z}}(\theta)=\left(\widetilde{R}^{\overline{z}\,\overline{z}}_{kl}(\theta): k, l\in \mathbb{Z}\right),
\;\; \widetilde{R}^{\overline{z}\,\overline{z}}_{kl}(\theta)=\frac{1}{2}\sum_{j\in \mathbb{Z}}
\frac{c_{jlk}v_{j}(\theta)}{\sqrt[4]{\lambda_{k}}\sqrt[4]{\lambda_{l}}}.
\end{equation}
For the sequence $z=(z_{j}\in \mathbb{C},\;j\in \mathbb{Z}),$ we can rewrite $z$ as
\begin{equation}\label{2*}
z=(z_{0}, \,z_{j}, \,z_{-j}:\;j=1, 2,\cdots)\triangleq u=(u_{j}: \;j=0, 1, 2,\cdots),
\end{equation}
where $u_{0}=z_{0}$, $u_{j}=(z_{j}, \,z_{-j})^{T},\;j=1, 2,\cdots.$
Here $(z_{j}, \,z_{-j})^{T}$ denotes the transpose of the vector $(z_{j}, \,z_{-j}).$
Let $\Lambda_{0}=\sqrt{\lambda_{0}},$ $\Lambda_{j}=\left(
                   \begin{array}{cc}
                     \sqrt{\lambda_{j}} & 0 \\
                     0 & \sqrt{\lambda_{-j}} \\
                   \end{array}
                 \right),\;\;j=1, 2, \cdots.$
Note that $\lambda_{j}=\lambda_{-j}=j^{2}+M,\;j=1, 2, \cdots.$
Then $\Lambda_{j}=\sqrt{\lambda_{j}}E_{22},\;j=1, 2, \cdots,$
 where $E_{22}$ is a $2\times 2$ unit matrix. For $u_{j}=(z_{j}, \,z_{-j})^{T}$
 and $\widetilde{u}_{j}=(\widetilde{z}_{j}, \,\widetilde{z}_{-j})^{T}, $
 define $u_{j}\cdot \widetilde{u_{j}}=z_{j}{\widetilde{z}}_{j}+z_{-j}{\widetilde{z}}_{-j},\;j=1, 2, \cdots.$\\
Then we can also rewrite \eqref{eq20} as
\begin{equation}\label{eq211}
\widetilde{H}=\langle\rho \Lambda u,\overline{u}\rangle+\varepsilon\mi \frac{b_0}{2}\Big(\langle \overline{u},\overline{u}\rangle-\langle u,u\rangle\Big)+ \varepsilon \left[\langle R^{uu}(\theta)u, u\rangle
+\langle R^{u\overline{u}}(\theta)u, \overline{u}\rangle+\langle R^{\overline{u}\overline{u}}(\theta)\overline{u}, \overline{u}\rangle\right],
\end{equation}
where
\begin{equation*}
\Lambda=diag \left({\Lambda_{j}}: j=0, 1, 2, \cdots\right),\;\;\theta=\omega t,
\end{equation*}
\begin{equation}\label{eq22.1}
R^{uu}(\theta)=\left(R^{uu}_{kl}(\theta): k, l=0, 1, 2, \cdots\right),\; R^{u\overline{u}}(\theta)=\left(R^{u\overline{u}}_{kl}(\theta): k, l=0, 1, 2, \cdots\right),
\end{equation}
\begin{equation}\label{eq22.2}
R^{\overline{u}\,\overline{u}}(\theta)=\left(R^{\overline{u}\,\overline{u}}_{kl}(\theta): k, l=0, 1, 2, \cdots\right),\;\;R^{uu}_{kl}(\theta)=R^{\overline{u}\,\overline{u}}_{kl}(\theta)=\frac{1}{2}R^{u\overline{u}}_{kl}(\theta),
\end{equation}
where
\begin{equation}\label{eq22.3}
R^{uu}_{kl}(\theta)=\left\{
                      \begin{array}{ll}
                        R_{0,0}(\theta), & {k=l=0;} \\
                        (R_{0,l}(\theta),\;R_{0,-l}(\theta)), & {k=0, \;l=1, 2, \cdots;} \\
                        (R_{k,0}(\theta),\;R_{-k,0}(\theta))^{T}, & {l=0, \;k=1, 2, \cdots;} \\
                        \left(
                          \begin{array}{cc}
                            R_{k,l}(\theta) & R_{k,-l}(\theta) \\
                            R_{-k,l}(\theta) & R_{-k,-l}(\theta) \\
                          \end{array}
                        \right), & {k, l=1, 2, \cdots,}
                      \end{array}
                    \right.
\end{equation}
and
 $$R_{k,l}(\theta)=\frac{1}{2}\sum_{j\in \mathbb{Z}}\frac{c_{jlk}v_{j}(\theta)}{\sqrt[4]{\lambda_{k}}\sqrt[4]{\lambda_{l}}},\;\; k, l\in \mathbb{Z}.$$
Define a Hilbert space $h_{\widetilde{N}}$ as follows:
\begin{equation}\label{eq231}
{h}_{\widetilde N}=\{x=(x_{k}\in\mathbb{C}:k\in \mathbb{Z}): \|x\|_{\widetilde{N}}^{2}=\sum_{k\in \mathbb{Z}}|k|^{2N}|x_{k}|^{2}\}.
\end{equation}
Similarly define a Hilbert space $h_{N}$ as follows:
\begin{equation}\label{eq23}
{h}_{ N}=\{y=(y_{k}:k=0, 1, \cdots): \|y\|_{{N}}^{2}=\sum_{k=0}^{\infty}|k|^{2N}|y_{k}|^{2}\},
\end{equation}
where $y_{0}\in \mathbb{C},$  $y_{k}=(z_{k}, \,z_{-k})^{T},\;z_{k}, \,z_{-k}\in \mathbb{C},\;k=1, 2,\cdots,$ and $|y_{k}|^{2}=|z_{k}|^{2}+|z_{-k}|^{2}.$
In \eqref{eq231} and \eqref{eq23}, we define $|k|^{2N}=1,$ if $k=0.$
For
$z=(z_{0}, \,z_{j}, \,z_{-j}:\;j=1, 2,\cdots)\in {h}_{\widetilde N},\;\;u=(u_{j}: \;j=0, 1, 2,\cdots)\in {h}_{ N},$
where $u_{0}=z_{0}$, $u_{j}=(z_{j}, \,z_{-j})^{T},\;j=1, 2,\cdots$. It can be obtained that
 $$\|u\|_{{N}}=\|z\|_{\widetilde{N}}.$$
%
\\
Recall that $$\mathcal{V}(\theta,x)\in {C}^{N}(\mathbb{T}^{n}\times [0, 2\pi], \mathbb{R}).$$
Note that the Fourier transformation \eqref{eq10} is isometric from $u\in\mathcal{H}^{N}[0, 2\pi]$
to $(u_{k}: k=0, 1, \cdots)\in h_{N},$ where $\mathcal{H}^{N}[0, 2\pi]$ is the usual Sobolev space.

Now we state a lemma, which is used in the next section.

\begin{lemma}
\begin{eqnarray}\label{eq25}
&&\sup_{\theta\in\mathbb{T}^{n}}\|\sum_{|\alpha|\leq N}\partial^{\alpha}_{\theta}J R^{uu}(\theta) J\|_{h_N\to h_N}\leq C,\nonumber\\
&& \sup_{\theta\in\mathbb{T}^{n}}\|\sum_{|\alpha|\leq N}\partial^{\alpha}_{\theta}J R^{u\overline{u}}(\theta) J\|_{h_N\to h_N}\leq C,\\
&& \sup_{\theta\in\mathbb{T}^{n}}\|\sum_{|\alpha|\leq N}\partial^{\alpha}_{\theta}J R^{\overline{u}\,\overline{u}}(\theta) J\|_{h_N\to h_N}\leq C\nonumber,
\end{eqnarray}
where $ ||\cdot||_{h_N\to h_N}$ is the operator norm from $h_N$ to $h_N$, and $\alpha=(\alpha_{1}, \alpha_{2}, \cdots , \alpha_{n}),$ $|\alpha|=|\alpha_{1}|+|\alpha_{2}|+\cdots+|\alpha_{n}|,$ $\alpha_{j}$'s
 are positive integers, and $J=diag (J_{j}: j=0, 1, \cdots),\;J_{0}=\sqrt[4]{\lambda_{0}},\;J_{j}=\sqrt[4]{\lambda_{j}}E_{22}, j=1, 2, \cdots.$
\end{lemma}
\begin{proof}
By \eqref{eq22.1}, \eqref{eq22.2} and \eqref{eq22.3}, we have that
$$\partial_{\theta}^{\alpha}{J}R^{uu}(\theta){J}\triangleq (A^{uu}_{kl}(\theta): k, l=0, 1, \cdots),$$
where $$A^{uu}_{kl}(\theta)=\left\{
                      \begin{array}{ll}
                        \frac{1}{2}\sum_{j\in \mathbb{Z}}c_{j\,0\,0}\partial_{\theta}^{\alpha}v_{j}(\theta), & {k=l=0;} \\
                        (\frac{1}{2}\sum_{j\in \mathbb{Z}}c_{j\,l\,0}\partial_{\theta}^{\alpha}v_{j}(\theta),\;\frac{1}{2}\sum_{j\in \mathbb{Z}}c_{j\,-l\,0}\partial_{\theta}^{\alpha}v_{j}(\theta)), & {k=0, \;l=1, 2, \cdots;} \\
                        (\frac{1}{2}\sum_{j\in \mathbb{Z}}c_{j\,0\,k}\partial_{\theta}^{\alpha}v_{j}(\theta),\;\frac{1}{2}\sum_{j\in \mathbb{Z}}c_{j\,0\,-k}\partial_{\theta}^{\alpha}v_{j}(\theta))^{T}, & {l=0, \;k=1, 2, \cdots;} \\
                        \left(
                          \begin{array}{cc}
                            \frac{1}{2}\sum_{j\in \mathbb{Z}}c_{j\,l\,k}\partial_{\theta}^{\alpha}v_{j}(\theta) & \frac{1}{2}\sum_{j\in \mathbb{Z}}c_{j\,-l\,k}\partial_{\theta}^{\alpha}v_{j}(\theta) \\
                            \frac{1}{2}\sum_{j\in \mathbb{Z}}c_{j\,l\,-k}\partial_{\theta}^{\alpha}v_{j}(\theta) & \frac{1}{2}\sum_{j\in \mathbb{Z}}c_{j\,-l\,-k}\partial_{\theta}^{\alpha}v_{j}(\theta) \\
                          \end{array}
                        \right), & {k, l=1, 2, \cdots.}
                      \end{array}
                    \right.
$$
For any $u=(u_{k}: k=0, 1, \cdots)\in h_{N},$
\begin{equation}\label{eq2.1}
\left(\sum_{|\alpha|\leq N}\partial_{\theta}^{\alpha}{J}R^{uu}(\theta){J}\right)u=\left(
\sum_{k=0}^{\infty}(\sum_{|\alpha|\leq N}A^{uu}_{lk})u_{k}: l=0, 1, \cdots\right).
\end{equation}
Suppose $\widetilde{J}=diag (\sqrt[4]{\lambda_{j}}: j\in \mathbb{Z}).$ Then for any $z=(z_{k}\in \mathbb{C}: k\in {Z})\in h_{\widetilde{N}},$
\begin{equation}\label{eq2.2}
\left(\sum_{|\alpha|\leq N}\partial_{\theta}^{\alpha}\widetilde{J}\,\widetilde{R}^{zz}(\theta)\widetilde{J}\right)z=\left(\frac{1}{2}
\sum_{j\in \mathbb{Z}}\sum_{k\in \mathbb{Z}}C_{jlk}(\sum_{|\alpha|\leq N}\partial_{\theta}^{\alpha}v_{j}(\theta))z_{k}: l\in \mathbb{{Z}}\right).
\end{equation}
A combination of \eqref{eq231}, \eqref{eq23}, \eqref{eq2.1} and \eqref{eq2.2} gives
\begin{eqnarray}\label{eq2}
&&\left\|\left(\sum_{|\alpha|\leq N}\partial_{\theta}^{\alpha}{J}R^{uu}(\theta){J}\right)u\right\|_{N}^{2}\nonumber\\
&=&\sum_{l=0}^{\infty}l^{2N}\left| \sum_{k=0}^{\infty}(\sum_{|\alpha|\leq N}A^{uu}_{lk})u_{k}\right|^{2}\nonumber\\
&=&\sum_{l\in \mathbb{Z}}|l|^{2N}\left| \frac{1}{2}\sum_{k\in \mathbb{Z}}\sum_{j\in \mathbb{Z}}
C_{jlk}(\sum_{|\alpha|\leq N}\partial_{\theta}^{\alpha}v_{j}(\theta))z_{k}\right|^{2}.
\end{eqnarray}
Let $$\gamma_{\,lj}=\frac{( l+j)j}{l}, \;\;\mbox{where} \;\; l, j = 1, 2, \cdots.$$
Note that $$c_{jlk}=
\left\{%
\begin{array}{ll}
 \;\;0,\;j+l-k\neq 0,\\
2\pi,\; j+l-k=0.\\
\end{array}%
\right.$$
By \eqref{eq2}, one has

\begin{eqnarray*}
&&\left\|\left(\sum_{|\alpha|\leq N}\partial_{\theta}^{\alpha}{J}R^{uu}(\theta){J}\right)u\right\|_{N}^{2}\\
&=&\sum_{l\in \mathbb{Z}}|l|^{2N}\left|\frac{1}{2}\sum_{j\in \mathbb{Z}} C_{jl( l+ j)}(\sum_{|\alpha|\leq N}\partial_{\theta}^{\alpha}v_{j}(\theta))z_{ l+ j}\right|^{2}
=\frac{1}{4}\left|\sum_{j\in \mathbb{Z}}c_{j0j}(\sum_{|\alpha|\leq N}\partial_{\theta}^{\alpha}v_{j}(\theta))z_{j}\right|^{2}\\
&&+\frac{1}{4}\sum_{l\in \mathbb{Z}\setminus \{0\}}|l|^{2N}\left|c_{0ll}(\sum_{|\alpha|\leq N}\partial_{\theta}^{\alpha}v_{0}(\theta))z_{l}+\sum_{j\in \mathbb{Z}\setminus \{0\}} C_{jl(l+j)}(\sum_{|\alpha|\leq N}\partial_{\theta}^{\alpha}v_{j}(\theta))z_{l+j}\right|^{2}\\
&\leq & C\left(\sum_{j\in \mathbb{Z}}|j|^{2N}\Big|\sum_{|\alpha|\leq N}\partial_{\theta}^{\alpha}v_{j}(\theta)\Big|^{2}\right)(\sum_{j\in \mathbb{Z}}|j|^{2N}|z_{j}|^{2})+C\sum_{l\in \mathbb{Z}\setminus \{0\}}|l|^{2N}\Big|\sum_{|\alpha|\leq N}\partial_{\theta}^{\alpha}v_{0}(\theta)
\Big|^{2}|z_{l}|^{2}\\
& &+C\sum_{l\in \mathbb{Z}\setminus \{0\}}\Big|\sum_{j\in \mathbb{Z}\setminus \{0\}}\frac{1}{\gamma_{\,lj}^{N}}\cdot \gamma_{\,lj}^{N}|l|^{N}(\sum_{|\alpha|\leq N}\partial_{\theta}^{\alpha}v_{j}(\theta))
z_{ l+ j}\Big|^{2}\\
&\leq&C||z||_{\widetilde{N}}^{2}+ C\sum_{l\in \mathbb{Z}\setminus \{0\}}\left(\sum_{j\in \mathbb{Z}\setminus \{0\}}\frac{1}{\gamma_{\,lj}^{2N}}\right)(\sum_{j\in \mathbb{Z}\setminus \{0\}}|j|^{2N}\Big|\sum_{|\alpha|\leq N}\partial_{\theta}^{\alpha}v_{j}(\theta)
\Big|^{2}|l+j|^{2N}|z_{l+j}|^{2})\\
&\leq &C||z||_{\widetilde{N}}^{2}+C\sum_{j\in \mathbb{Z}\setminus \{0\}}|j|^{2N}\Big|\sum_{|\alpha|\leq N}\partial_{\theta}^{\alpha}v_{j}(\theta)\Big|^{2}\|z\|_{N}^{2}\\
&\leq &C||z||_{\widetilde{N}}^{2}+C\sup_{(\theta,x)\in \mathbb{T}^{n}\times [0, 2\pi]}
\big|\sum_{|\alpha|\leq N}\partial_{\theta}^{\alpha}\partial_{x}^{N}\mathcal{V}(\theta, x)\big|\|z\|_{N}^{2}\\
&\leq &C \|z\|_{\widetilde{N}}^{2}=C \|u\|_{{N}}^{2},
\end{eqnarray*}
where $C$ is a universal constant which might be different in different places.  It follows that
\begin{equation}\label{eq2.26}
\sup_{\theta\in\mathbb{T}^{n}}\|\sum_{|\alpha|\leq N}\partial^{\alpha}_{\theta}J R^{u{u}}(\theta) J\|_{h_N\to h_N}\leq C.
\end{equation}
The proofs of the last two inequalities in \eqref{eq25} are similar to that of \eqref{eq2.26}.
\end{proof}
Now our goal is to find a symplectic transformation $\Psi$, such that the term $\varepsilon \mi \frac{b_0}{2}\Big( \langle u,u\rangle+\langle \overline{u},\overline{u}\rangle\Big)$ disappear. To this end, let $G$ be a linear Hamiltonian of the form
\begin{equation}
G=b_1(\theta)\Big(\langle \Lambda^{-1}u,u\rangle+\langle\Lambda^{-1}\bar{u},\bar{u}\rangle \Big),
\end{equation}
where $\theta=\omega t$ and $b_{1}(\theta)$ need to be specified. Moreover, let
\begin{equation}
\Psi=X^t_{\varepsilon G}|_{t=1},
\end{equation}
where $X^t_{\varepsilon G}$ is the flow of Hamiltonian, $X_{\varepsilon G}$ is the vector field  of the Hamiltonian $\varepsilon G$ with the symplectic $\mi du \wedge d\bar{u}$. Let
\begin{equation}
H_0=\widetilde{H}\circ \Psi.
\end{equation}
Recall that $$
\widetilde{H}=\langle\rho \Lambda u,\overline{u}\rangle+\varepsilon\mi \frac{b_0}{2}\Big(\langle \overline{u},\overline{u}\rangle-\langle u,u\rangle\Big)+ \varepsilon \left[\langle R^{uu}(\theta)u, u\rangle
+\langle R^{u\overline{u}}(\theta)u, \overline{u}\rangle+\langle R^{\overline{u}\overline{u}}(\theta)\overline{u}, \overline{u}\rangle\right].
$$
Then we have $\widetilde{H}=N+\varepsilon Q+\varepsilon R_{0}$, where
\begin{equation}
 N=\langle\rho \Lambda u,\overline{u}\rangle, \quad Q=\mi\frac{b_0}{2}\Big( \langle \overline{u},\overline{u}\rangle-\langle u,u\rangle\Big),
\end{equation}
\begin{equation}
R_{0}=\left[\langle R^{uu}(\theta)u, u\rangle
+\langle R^{u\overline{u}}(\theta)u, \overline{u}\rangle+\langle R^{\overline{u}\overline{u}}(\theta)\overline{u}, \overline{u}\rangle\right].
\end{equation}
 Since the Hamiltonian $\widetilde{H}=\widetilde{H}(\omega t, u, \overline{u})$ depends on time $t,$ we introduce a fictitious
 action $I=$ constant, and let $\theta=\omega t$ be angle variable. Then the non-autonomous $\widetilde{H}(\omega t, u, \overline{u})$ can be written as
 $$\omega I+\widetilde{H}(\theta, u, \overline{u})$$
 with symplectic structure $d I\wedge d\theta +\mi\,  d u\wedge d \overline{u}$. See Section 45 (B) in \cite{Arnold}.
By Taylor formula, we have
\begin{equation}
\begin{split}
H_0=&\widetilde{H}\circ X^1_{\varepsilon G}\\
=&N+\varepsilon Q+\varepsilon\{N,G\}+\varepsilon^2\int^{1}_0\{Q,G\}\circ X^\tau_{\varepsilon G}d\tau \\
&+ \varepsilon^2\int^{1}_0(1-\tau)\{\{N,G\},G\}\circ X^\tau_{\varepsilon G}d\tau+\varepsilon R_{0}\circ X^1_{\varepsilon G}.
\end{split}
\end{equation}
where $\{N,G\}=\omega \cdot \partial_{\theta} b_1\Big(\langle\Lambda^{-1} u,u\rangle+ \langle\Lambda^{-1} \bar{u},\bar{u}\rangle\Big)-i \rho b_1 \Big(\langle \bar{u},\bar{u}\rangle-\langle u,u\rangle\Big).$
Let $ b_1=\frac{b_0}{2\rho}$, then we have $H_{0}=N+R$,
 where
 \begin{eqnarray}
 R&=&\varepsilon \omega \cdot \partial_{\theta} b_1\Big(\langle\Lambda^{-1} u,u\rangle- \langle\Lambda^{-1} \bar{u},\bar{u}\rangle\Big)\label{eq2.24}\\
 &+&\varepsilon^2\int^{1}_0\{{Q},G\}\circ X^\tau_{\varepsilon G}d\tau\label{eq2.27}\\
 &+& \varepsilon^2\int^{1}_0(1-\tau)\{\{N,G\},G\}\circ X^\tau_{\varepsilon G}d\tau\label{eq2.28}\\
  &+&\varepsilon R_{0}\circ X^1_{\varepsilon G}.\label{eq2.25}
  \end{eqnarray}
The aim of following section is to estimate $R$.

$\bullet$ Estimate of \eqref{eq2.24}.

Let
\begin{equation*}
\overline{G}^{*}=\left(
            \begin{array}{cc}
            \frac{\omega\cdot \partial_{\theta}b_0}{2\rho}\Lambda^{-1} & 0 \\
            0 & -\frac{\omega\cdot \partial_{\theta}b_0}{2\rho}\Lambda^{-1}\\
            \end{array}
          \right),
\;\;\widetilde{u}=\left(
                       \begin{array}{c}
                         u \\
                         \overline{u} \\
                       \end{array}
                     \right).
\end{equation*}
Then, we have \eqref{eq2.24}$=\langle \varepsilon \overline{G}^*\widetilde{u},\widetilde{u}\rangle$. Obviously, $$\sup_{\theta\in\mathbb{T}^{n}}\|\sum_{|\alpha|\leq {N-1}}\partial^{\alpha}_{\theta}J \overline{G}^{*}(\theta) J\|_{h_{N}\to h_{N}}\leq C.$$

$\bullet$ Estimate of \eqref{eq2.25}.
Let
\begin{equation*}
\widehat{R}=\left(
            \begin{array}{cc}
             R^{uu}(\theta, \omega) & \frac{1}{2}R^{u\,\overline{u}}(\theta, \omega) \\
            \frac{1}{2} R^{u\overline u}(\theta, \omega) & R^{\overline {u}\overline{u}}(\theta, \omega)\\
            \end{array}
          \right),
\;\;\mathcal{J}=\left(\begin{array}{cc}
             0& -\mi\, id \\
            \mi\, id & 0\\
            \end{array}
          \right).
\end{equation*}
and
\begin{equation}
\overline{G}=\left(
            \begin{array}{cc}
            \frac{b_0}{2\rho}\Lambda^{-1} & 0 \\
            0 & -\frac{b_0}{2\rho}\Lambda^{-1}\\
            \end{array}
          \right).
\end{equation}
Then we have
$$R_{0}=\langle \widehat{R}(\theta)\left(
                       \begin{array}{c}
                         u \\
                         \overline{u} \\
                       \end{array}
                     \right),
                     \left(
                       \begin{array}{c}
                         u \\
                         \overline{u} \\
                       \end{array}
                     \right)\rangle.
$$
It follows that
\begin{equation}\label{eq9.2}
\varepsilon^{2}\{R_{0}, G\}
=4\varepsilon^{2} \langle {\widehat{R}}(\theta)\mathcal{J}\overline{G}(\theta)\widetilde{u}, \widetilde{u}\rangle.
\end{equation}
Let $\widehat{G}=\mathcal{J}\overline{G}(\theta)$ and $[\widehat{R}, \widehat{G} ]=\widehat{R}\widehat{G}+(\widehat{R}\widehat{G})^{T}$. By Taylor formula, we have
$$\eqref{eq2.25}=\varepsilon\langle R_1^{*}\tilde{u},\tilde{u} \rangle,$$
where
\begin{equation}
R_1^{*}=\widehat{R}+2^{2}\varepsilon\widehat{R}\widehat{G}
+\sum_{j=2}^{\infty}\frac{2^{j+1}\varepsilon^{j}}{j!}\underbrace{[\cdots [\widehat{R},\widehat{G}], \cdots , \widehat{G}]}_{j-1-\mbox{fold}}\widehat{G}.
\end{equation}
Thus, we can see
$$\sup_{\theta\in\mathbb{T}^{n}}\|\sum_{|\alpha|\leq {N}}\partial^{\alpha}_{\theta}J R_1^{*}(\theta) J\|_{h_{N}\to h_{N}}\leq C.$$

$\bullet$ Estimate of \eqref{eq2.27}.
\begin{equation}
\{ {Q},G\}=\frac{2b^2_0}{\rho}\langle \Lambda^{-1}u,\bar{u} \rangle= \langle K^{*}\widetilde{u}, \widetilde{u}\rangle,
\end{equation}
where
\begin{equation}
K_1^{*}=\left(
            \begin{array}{cc}
            0 & \frac{2b^2_0}{\rho}\Lambda^{-1} \\
            \frac{2b^2_0}{\rho}\Lambda^{-1} & 0 \\
            \end{array}
          \right).
\end{equation}
By Taylor formula, we have
$$\eqref{eq2.27}=\varepsilon^2 \langle K^{*}\tilde{u},\tilde{u} \rangle$$
where
\begin{equation}
K^{*}=K^{*}_1
+\sum_{j=2}^{\infty}\frac{2^{j-1}\varepsilon^{j}}{j!}\underbrace{[\cdots K^{*}_1, \cdots , \widehat{G}]}_{j-2-\mbox{fold}}\widehat{G}.
\end{equation}
Now we have
$$\sup_{\theta\in\mathbb{T}^{n}}\|\sum_{|\alpha|\leq {N-1}}\partial^{\alpha}_{\theta}J K_1^{*}(\theta) J\|_{h_{N}\to h_{N}}\leq C.$$

$\bullet$ Estimate of \eqref{eq2.28}.

By  directly calculation, we have
\begin{equation}
\{\{N,G\},G\}=\langle H^{*}_1\bar{u},\bar{u} \rangle,
\end{equation}
where
\begin{equation}
H_1^{*}=\left(
            \begin{array}{cc}
            0 & -\frac{b^2_0}{\rho}\Lambda^{-1}  \\
            -\frac{b^2_0}{\rho}\Lambda^{-1}  & 0\\
            \end{array}
          \right).
\end{equation}
By Taylor formula, we have
$$\eqref{eq2.28}=\varepsilon^2 \langle H^{*}\tilde{u},\tilde{u} \rangle,$$
where
\begin{equation}
H^{*}=\frac{K^{*}_1}{2}
+\sum_{j=3}^{\infty}\frac{2^{j-2}\varepsilon^{j-1}}{j!}\underbrace{[\cdots H^{*}_1, \cdots , \widehat{G}]}_{j-3-\mbox{fold}}\widehat{G}.
\end{equation}
Now we have
$$\sup_{\theta\in\mathbb{T}^{n}}\|\sum_{|\alpha|\leq {N}}\partial^{\alpha}_{\theta}J H_1^{*}(\theta) J\|_{h_{N}\to h_{N}}\leq C.$$
In conclusion,
$$\sup_{\theta\in\mathbb{T}^{n}}\|\sum_{|\alpha|\leq {N-1}}\partial^{\alpha}_{\theta}J R J\|_{h_{N}\to h_{N}}\leq C.$$
Now, Theorem 1.1 can be transformed into a more exact expression.
\begin{theorem}\label{thm2.1}
With Assumptions {\bf A, B}, for given $1\gg\gamma>0$,
 there exists $\epsilon^*$ with $ \;0<\varepsilon^{*}=\varepsilon^{*} (n, \gamma)\ll \gamma,$
and exists a subset $\Pi\subset [1, 2]^{n}$ with
$$\mbox{Measure}\, \Pi\geq 1-O (\gamma^{1/3})$$
such that for any $0<\varepsilon<\varepsilon^{*}$ and any $\omega \in \Pi,$ there is a time-quasi-periodic  symplectic change
$$\left(
    \begin{array}{c}
      u \\
      \overline{u} \\
    \end{array} \right)=\Phi(\omega t)\left(
                                \begin{array}{c}
                                  \widetilde{{u}} \\
                                  \overline{ \widetilde{{u}}} \\
                                \end{array}
                              \right)
$$
such that the Hamiltonian system \eqref{eq211} is changed into
\begin{equation*}
\left\{%
\begin{array}{cl}
&\dot{\widetilde{u}}_{k}=\mi\, \frac{\partial \widetilde{H}}{\partial \overline{\widetilde{u}}_{k}},\;\;k\in\mathbb{Z},\\
\\
&\dot{\overline{\widetilde{u}}}_{k}=-\mi\, \frac{\partial \widetilde{H}}{\partial \widetilde{u}_{k}},\;\;k\in\mathbb{Z},\\
\end{array}%
\right.
\end{equation*}
where $$\widetilde{H}(\widetilde{u}, \overline{\widetilde{u}})=\Lambda_{0}^{\infty}\widetilde{u}_{0}\overline{\widetilde{u}}_{0}
+\sum_{j=1}^{\infty}(\Lambda_{j}^{\infty}\widetilde{{u}}_{j})\cdot \overline{\widetilde{{u}}}_{j},$$
where
$$\Lambda_{0}^{\infty}=\sqrt{\lambda_{0}}+\varepsilon Q_{0},\;\;\Lambda_{j}^{\infty}=\sqrt{\lambda_{j}}E_{22}+\varepsilon Q_{j}$$
with
\begin{description}
  \item[(i)] $Q_{0}$ and $Q_{k}\,(k=1, 2, \cdots)$ are independent of time $t,$ and $Q_{0}\in \mathbb{R},$ $Q_{k}$ is a $2 \times 2$  real matrix $(k=1, 2, \cdots)$;
  \item[(ii)] $\widetilde{Q}=diag (Q_{j})$ satisfies $\|J\widetilde{Q}J\|_{h_{N}\rightarrow h_{N}}\leq C,\;$ $J=diag (J_{j}: j=0, 1, \cdots),\;J_{0}=\sqrt[4]{\lambda_{0}},\;J_{j}=\sqrt[4]{\lambda_{j}}E_{22}, j=1, 2, \cdots$;
  \item[(iii)] $\Phi=\Phi(\omega t)$ is quasi-periodic in time and close to the identity map:
  \[\|\Phi(\omega t)-id\|_{h_{N}\rightarrow h_{N}}\leq C\varepsilon,\]where $id$ is the identity map from $h_{N}\rightarrow h_{N}.$
\end{description}
\end{theorem}
\section{Analytical Approximation Lemma}
We need to find a series of operators which are analytic in some complex strip domains to approximate the operators
 $R^{uu}(\theta), R^{u\overline{u}}(\theta)$ and $R^{\overline{u}\,\overline{u}}(\theta)$. To this end, we cite an approximation lemma
(see \cite{Jackson, Salamon1989, Salamon2004} for the details). This method is used in \cite{yuan-zhang}, too.

We start by recalling some definitions and setting some new notations. Assume $X$ is a Banach space with the norm
$||\cdot||_{X}$. First recall that $C^{\mu}(\mathbb{R}^{n}; X)$ for $0< \mu <1$ denotes the space of bounded
H\"{o}lder continuous functions $f: \mathbb{R}^{n}\mapsto X$ with the form
$$\|f\|_{C^{\mu}, X}=\sup_{0<|x-y|<1}\frac{\|f(x)-f(y)\|_{X}}{|x-y|^{\mu}}+\sup_{x\in \mathbb{R}^{n}}\|f(x)\|_{X}.$$
If $\mu=0$ then $\|f\|_{C^{\mu},X}$ denotes the sup-norm. For $\ell=k+\mu$ with $k\in \mathbb{N}$ and $0\leq \mu <1,$
we denote by $C^{\ell}(\mathbb{R}^{n};X)$ the space of functions $f:\mathbb{R}^{n}\mapsto X$ with H\"{o}lder continuous partial derivatives, i.e., $\partial ^{\alpha}f\in C^{\mu}(\mathbb{R}^{n}; X_{\alpha})$ for all muti-indices $\alpha=(\alpha_{1}, \cdots, \alpha_{n})\in \mathbb{N}^{n}$ with the assumption that
$|\alpha|:=|\alpha_{1}|+\cdots+|\alpha_{n}|\leq k$ and $X_{\alpha}$ is the Banach space of bounded operators
$T:\prod^{|\alpha|}(\mathbb{R}^{n})\mapsto X$ with the norm
$$\|T\|_{X_{\alpha}}=\sup \{||T(u_{1}, u_{2}, \cdots, u_{|\alpha|})||_{X}:\|u_{i}\|=1, \;1\leq i \leq |\alpha|\}.$$
We define the norm
$$||f||_{C^{\ell}}=\sup_{|\alpha|\leq \ell}||\partial ^{\alpha}f||_{C^{\mu}, X_{\alpha}}$$
\begin{lemma}(Jackson-Moser-Zehnder)\label{Jackson}
Let $f\in C^{\ell}(\mathbb{R}^{n}; X)$ for some $\ell>0$ with finite $C^{\ell}$ norm over $\mathbb{R}^{n}.$
Let $\phi$ be a radial-symmetric, $C^{\infty}$ function, having as support the closure of the unit ball centered at the origin, where $\phi$ is completely flat and takes value 1. Let $K=\widehat{\phi}$ be its Fourier transform. For all $\sigma >0$ define
$$ f_{\sigma}(x):=K_{\sigma}\ast f=\frac{1}{\sigma^{n}}\int_{\mathbb{R}^{n}}K(\frac{x-y}{\sigma})f(y)dy.$$
Then there exists a constant $C\geq 1$ depending only on $\ell$ and $n$ such that the following holds: for any $\sigma >0,$ the function $f_{\sigma}(x)$ is a real-analytic function from $\mathbb{C}^{n}/(\pi \mathbb{Z})^{n}$ to $X$ such that if $\Delta_{\sigma}^{n}$ denotes the $n$-dimensional complex strip of width $\sigma,$
$$\Delta_{\sigma}^{n}:=\{x\in \mathbb{C}^{n}\big ||\mathrm{Im} x_{j}|\leq \sigma,\;1\leq j\leq n\},$$
then for any $\alpha\in\mathbb{N}^{n}$ such that $|\alpha|\leq \ell$ one has

\begin{equation*}\label{cite3.1}
\sup_{x\in \Delta_{\sigma}^{n}}||\partial ^{\alpha}f_{\sigma}(x)
-\sum_{|\beta|\leq \ell-|\alpha|}\frac{\partial^{\beta+\alpha}f(\mathrm{Re}x)}{\beta !}(\sqrt{-1}\mathrm{Im}x)^{\beta}||_{X_{\alpha}}\leq C ||f||_{C^{\ell}}\sigma^{\ell-|\alpha|},
\end{equation*}

and for all $0\leq s\leq \sigma,$
\begin{equation*}\label{cite3.2}
\sup_{x\in \Delta_{s}^{n}}\|\partial^{\alpha} f_{\sigma}(x)-\partial^{\alpha}f_{s}(x)\|_{X_{\alpha}}
\leq C ||f||_{C^{\ell}}\sigma^{\ell-|\alpha|}.
\end{equation*}

The function $f_{\sigma}$ preserves periodicity (i.e., if $f$ is T-periodic in any of its variable $x_{j}$, so is $f_{\sigma}$). Finally, if $f$ depends on some parameter $\xi\in \Pi\subset\mathbb{R}^{n}$ and
$$||f(x,\xi)||_{C^{\ell}(X)}^{\mathcal{ L}}:=\sup_{\xi\in\Pi}||\partial_{\xi}\,f(x,\xi)|
|_{C^{\ell}(X)}$$ are uniformly bounded by a constant $C$,
then all the above estimates hold true with $\|\cdot\|$ replaced by $\|\cdot\|^{\mathcal{L}}.$
\end{lemma}
The proof of this lemma consists in a direct check which is based on standard tools from calculus and complex analysis. It is used to deal with KAM theory for finite smooth systems by Zehnder \cite{E. Zehnder}. Also see \cite{Chierchia} and \cite{yuan-zhang} and references therein, for example. For simplicity of notation, we shall replace $\|\cdot\|_{X}$ by
$\|\cdot\|.$ Now let us apply this lemma to the perturbation $P(\phi).$

Fix a sequence of fast decreasing numbers $s_{\nu}\downarrow 0, \upsilon\geq 0,$ and $s_{0}\leq \frac{1}{2}.$
For an $X$-valued function $P(\phi),$
construct a sequence of real analytic functions $P^{(\upsilon)}(\phi)$ such that the following conclusions hold:
\begin{description}
  \item[(1)] $P^{(\upsilon)}(\phi)$ is real analytic on the complex strip $\mathbb{T}^{n}_{s_{\upsilon}}$ of the width $s_{\upsilon}$ around $\mathbb{T}^{n}.$
  \item[(2)] The sequence of functions $P^{(\upsilon)}(\phi)$ satisfies the bounds:
  \begin{equation}\label{cite3.3}
  \sup_{\phi\in\mathbb{T}^{n}}\|P^{(\upsilon)}(\phi)-P(\phi)\|\leq C \|P\|_{C^{\ell}}s_{\upsilon}^{\ell},
  \end{equation}
  \begin{equation}\label{cite3.4}
  \sup_{\phi\in\mathbb{T}^{n}_{s_{\upsilon+1}}}\|P^{(\upsilon+1)}(\phi)-P^{(\upsilon)}(\phi)\|\leq C \|P\|_{C^{\ell}}s_{\upsilon}^{\ell},
  \end{equation}
  where $C$ denotes (different) constants depending only on $n$ and $\ell.$
  \item[(3)] The first approximate $P^{(0)}$ is ``small" with the perturbation $P$. Precisely speaking, for arbitrary $\phi\in\mathbb{T}^{n}_{s_{0}},$ we have
      \begin{eqnarray}\label{cite3.5}
      \|P^{(0)}(\phi)\|&\leq &
      C\|P\|_{C^{\ell}},
      \end{eqnarray}
      where the constant $C$ is independent of $s_{0},$ and the last inequality holds true due to the hypothesis that
      $s_{0}\leq \frac{1}{2}.$
  \item[(4)] From the first inequality \eqref{cite3.3}, we have the equality below. For any arbitrary $\phi\in\mathbb{T}^{n},$
      \begin{equation}\label{cite3.6}
      P(\phi)=P^{(0)}(\phi)+\sum_{\upsilon=0}^{+\infty}(P^{(\upsilon+1)}(\phi)-P^{(\upsilon)}(\phi)).
      \end{equation}
\end{description}

Now take a sequence of real numbers $\{s_{v}\geq 0\}_{v=0}^{\infty}$ with $s_{v}>s_{v+1}$ going fast to zero.
Let $R^{p,q}(\theta)=P(\theta)$ for $p, q\in \{u, \overline{u}\}.$ Then by \eqref{cite3.6} and \eqref{eq25},
 for $p, q\in\{u, \overline{u}\},$ we have,
  \begin{equation}\label{*}
 R^{p,q}(\theta)=R_{0}^{p,q}(\theta)+\sum_{l=1}^{\infty}R^{p,q}_{l}(\theta),
 \end{equation}
 where $R^{p,q}_{0}(\theta)$ is analytic in $\mathbb{T}_{s_{0}}^{n}$ with
  \begin{equation}\label{**}
 \sup_{\theta\in\mathbb{T}^{n}_{s_{0}}}\|R_{0}^{p,q}(\theta)\|_{h_{N}\to h_{N}}\leq C,
  \end{equation}
 and $R_{l}^{p,q}(\theta)\;(l\geq 1)$ is analytic in $\mathbb{T}^{n}_{s_{l}}$ with
 \begin{equation}\label{***}
 \sup_{\theta\in\mathbb{T}^{n}_{s_{l}}}\|J R_{l}^{p,q}(\theta) J\|_{h_{N}\to h_{N}}\leq Cs^{N}_{l}.
  \end{equation}
 \section{Iterative parameters of domains}
 Let

\begin{itemize}
  \item $\varepsilon_{0}=\varepsilon, \varepsilon_{\nu}=\varepsilon^{(\frac{4}{3})^{\nu}}, \nu=0, 1, 2, \cdots,$
  which measures the size of perturbation at $\nu-th$ step.
  \item $s_{\nu}=\varepsilon_{\nu+1}^{1/N}, \nu=0, 1, 2, \cdots,$
  which measures the strip-width of the analytic domain $\mathbb{T}_{s_{\nu}}^{n},$
  $\mathbb{T}_{s_{\nu}}^{n}=\{\theta\in\mathbb{C}^{n}/2\pi\mathbb{Z}^{n}: |Im \theta|\leq s_{\nu}\}.$
  \item $C({\nu})$ is a constant which may be different in different places, and it is of the form
   $$C({\nu})= C_{1} 2^{C_{2}\nu},$$
  where $C_{1},$ $C_{2}$ are  constants.
  \item $K_{\nu}=100 s_{\nu}^{-1}2^{\nu}|\log \varepsilon|.$
  \item$\gamma_{\nu}=\frac{\gamma}{2^{\nu}},\,0<\gamma\ll 1.$
  \item A family of subsets $\Pi_{\nu}\subset [1, 2]^{n}$ with $[1, 2]^{n}\supset \Pi_{0}\supset \cdots \supset \Pi_{\nu}\supset \cdots,$
  and $$mes \Pi_{\nu}\geq mes \Pi_{\nu-1}-C\gamma^{1/3}_{\nu-1}.$$

  \item For an operator-value (or a vector-value) function $B(\theta, \omega),$ whose domain is
  $(\theta, \omega)\in \mathbb{T}_{s_{\nu}}^{n}\times \Pi_{\nu}.$ Set
  $$ \|B\|_{\mathbb{T}^{n}_{s_{\nu}}\times \Pi_{\nu}}=\sup_{(\theta, \omega)
  \in \mathbb{T}^{n}_{s_{\nu}}\times \Pi_{\nu}}\|B(\theta, \omega)\|_{h_{N}\to h_{N}},$$ where $\|\cdot\|_{h_{N}\to h_{N}}$ is the operator norm, and set
  $$ \|B\|^{\mathcal{L}}_{\mathbb{T}^{n}_{s_{\nu}}\times \Pi_{\nu}}=\sup_{(\theta, \omega)
  \in \mathbb{T}^{n}_{s_{\nu}}\times \Pi_{\nu}}\|\partial_{\omega}B(\theta, \tau)\|_{h_{N}\to h_{N}}.$$
\end{itemize}
\section{Iterative Lemma}
In the following, for a function $f(\omega)$, denote by $\partial_{\omega}$ the derivative of $f(\omega)$ with respect to $\omega$ in Whitney's sense.
\begin{lemma}
For $p, q\in \{u, \overline{u} \},$ let $R_{0,0}^{p,q}=R_{0}^{p,q}, R_{l,0}^{p,q}=R_{l}^{p,q},$ where
$R_{0}^{p,q}, R_{l}^{p,q}$ are defined by \eqref{*}, \eqref{**} and \eqref{***}.
Assume that we have a family of Hamiltonian functions $H_{\nu}$:
\begin{equation}\label{eq26}
H_{\nu}=\Lambda_{0}^{(\nu)}u_{0}\overline{u}_{0}+\sum_{j=1}^{\infty}(\Lambda_{j}^{(\nu)}u_{j})\cdot\overline{u}_{j}
+\sum_{l\geq \nu}^{\infty}\varepsilon_{l}(\langle R^{uu}_{l,\nu}u,u\rangle+\langle R^{u\overline{u}}_{l,\nu}u, \overline{u}\rangle
+\langle R^{\overline{u}\,\overline{u}}_{l,\nu}\overline{u}, \overline{u}\rangle),\;\nu=0,1, \cdots, m,
\end{equation}
where $R^{uu}_{l,\nu}, R^{u\overline{u}}_{l,\nu}, R^{\overline{u}\overline{u}}_{l,\nu}$ are operator-valued functions defined on the domain $\mathbb{T}_{s_{\nu}}^{n}\times \Pi_{\nu},$ and
\begin{equation*}\label{eq27}
\theta=\omega t,\;\;\omega=(\omega_{1}, \omega_{2}, \cdots\omega_{n}).
\end{equation*}
\begin{description}
  \item[$(A1)_{\nu}$]
\begin{eqnarray}
  &&\Lambda_{0}^{(0)}=\rho\sqrt{\lambda_{0}},\;\;
\Lambda_{0}^{(\nu)}=\rho\sqrt{\lambda_{0}}+\sum_{i=0}^{\nu-1}\varepsilon_{i}\mu_{0}^{(i)},\;\;\nu\geq 1; \label{eq28}\\
&&\Lambda_{j}^{(0)}=\rho\sqrt{\lambda_{j}}E_{22},\;\;
\Lambda_{j}^{(\nu)}=\rho\sqrt{\lambda_{j}}E_{22}+\sum_{i=0}^{\nu-1}\varepsilon_{i}\mu_{j}^{(i)},\;\; j=1, 2, \cdots,\;\;\nu\geq 1,\label{eq281}
\end{eqnarray}
where
\begin{description}
  \item[(i)] $\mu_{0}^{(i)}=\mu_{0}^{(i)}(\omega): \Pi_{i}\rightarrow \mathbb{R}$ with
\begin{eqnarray}
 &&|\mu_{0}^{(i)}|_{\Pi_{i}}:=\sup_{\omega\in\Pi_{i}}|\mu_{0}^{(i)}(\omega)|\leq C(i),\;
  0\leq i\leq\nu-1,\label{eq291}\\ &&|\mu_{0}^{(i)}|_{\Pi_{i}}^{\mathcal{L}}:=\sup_{\omega\in\Pi_{i}}\max_{1\leq l\leq n}|\partial_{\omega_{l}}\mu_{0}^{(i)}(\omega)|\leq C(i),\;
  0\leq i\leq\nu-1.\label{eq0291}
  \end{eqnarray}
Here $|\cdot|$ denotes the absolute value of a function,
  \item[(ii)]$\mu_{j}^{(i)}=\mu_{j}^{(i)}(\omega)\;(j=1, 2, \cdots,\;0\leq i\leq\nu-1,\;\nu\geq 1)$ are $2\times 2$ real symmetry matrices with
\begin{eqnarray}
 &&|\mu_{j}^{(i)}|_{\Pi_{i}}:=\sup_{\omega\in\Pi_{i}}|\mu_{j}^{(i)}(\omega)|\leq C(i)/j, \label{eq29}\\
  &&|\mu_{j}^{(i)}|_{\Pi_{i}}^{\mathcal{L}}:=\sup_{\omega\in\Pi_{i}}\max_{1\leq l\leq n}|\partial_{\omega_{l}}\mu_{j}^{(i)}(\omega)|\leq C(i)/j.
  \label{eq029}
  \end{eqnarray}
Here $|\cdot|$ denotes the sup-norm for real matrices.
\end{description}

 \item[$(A2)_{\nu}$]
 For $p, q\in \{u,\overline{u}\},$ $R_{l,\nu}^{p,q}=R_{l,\nu}^{p,q}(\theta, \omega)$ is defined in
  $\mathbb{T}^{n}_{s_{l}}\times \Pi_{\nu}$ with $l\geq \nu,$ and is analytic in $\theta$ for fixed $\omega\in \Pi_{\nu},$
  and
  \begin{equation}\label{eq30}
  \|J R^{p,q}_{l,\nu} J\|_{\mathbb{T}^{n}_{s_{l}}\times \Pi_{\nu}}\leq C(\nu),
  \end{equation}
  \begin{equation}\label{eq31}
  \|J R^{p,q}_{l,\nu} J\|^{\mathcal{L}}_{\mathbb{T}^{n}_{s_{l}}\times \Pi_{\nu}}\leq C(\nu).
  \end{equation}
\end{description}
Then there exists a compact set $\Pi_{m+1}\subset \Pi_{m}$ with
 \begin{equation}\label{eq32}
 mes \Pi_{m+1}\geq mes \Pi_{m}-C\gamma_{m}^{1/3},
  \end{equation}
  and exists a symplectic coordinate change
  \begin{equation*}\label{eq33}
\Psi_{m}: \mathbb{T}^{n}_{s_{m+1}}\times \Pi_{m+1}\rightarrow \mathbb{T}^{n}_{s_{m}}\times \Pi_{m},
  \end{equation*}
  \begin{equation}\label{eq33H}
||\Psi_{m}-id ||_{h_{N}\to h_{N}}\leq \varepsilon^{1/2},\;(\theta,\omega)\in \mathbb{T}^{n}_{s_{m+1}}\times \Pi_{m+1}
  \end{equation}
  such that the Hamiltonian function $H_{m}$ is changed into
 \begin{eqnarray}\label{eq34}
H_{m+1}& \triangleq & H_{m}\circ \Psi_{m}\nonumber\\
&=&\Lambda_{0}^{(m+1)}u_{0}\overline{u}_{0}+\sum_{j=1}^{\infty}(\Lambda_{j}^{(m+1)}u_{j})\cdot\overline{u}_{j}
+\sum_{l\geq m+1}^{\infty}\varepsilon_{l}\left [\langle R^{uu}_{l,m+1}u,u\rangle\right.\\
&&\left.+\langle R^{u\overline{u}}_{l,m+1}u, \overline{u}\rangle
+\langle R^{\overline{u}\,\overline{u}}_{l,m+1}\overline{u}, \overline{u}\rangle\right ]\nonumber,
  \end{eqnarray}
  which is defined on the domain $\mathbb{T}^{n}_{s_{m+1}}\times \Pi_{m+1},$
  and ${\Lambda_{j}^{(m+1)}}^{,}s$ satisfy the assumptions $(A 1)_{m+1}$ and
  $R_{l, m+1}^{p,q} (p,q\in \{u, \overline{u}\})$ satisfy the assumptions $(A 2)_{m+1}.$
\end{lemma}
\section{Derivation of homological equations}
Our end is to find a symplectic transformation $\Psi_{\nu}$ such that the terms ${R^{uu}_{l, v}}$, ${R^{u\overline{u}}_{l, v}}$, ${R^{\overline{u}\,\overline{u}}_{l, v}}$
(with $l=v$) disappear. To this end, let $F$ be a linear Hamiltonian of the form
 \begin{equation}\label{eq35}
F=\langle F^{uu}(\theta,\omega)u, u\rangle+\langle F^{u \overline{u}}(\theta, \omega)u, \overline{u}\rangle+\langle F^{\overline{u}\,\overline{u}}(\theta, \omega)\overline{u}, \overline{u}\rangle,
  \end{equation}
where $\theta=\omega t,$ $(F^{uu}(\theta, \omega))^{T}=F^{uu}(\theta, \omega),$ $(F^{u\overline{u}}(\theta, \omega))^{T}=F^{u\overline{u}}(\theta, \omega),$ $(F^{\overline{u}\overline{u}}(\theta, \omega))^{T}=F^{\overline{u}\overline{u}}(\theta, \omega).$
Moreover, let
\begin{equation}\label{eq36}
\Psi=\Psi_{m}=X_{\varepsilon_{m}F}^{t}\big | _{t=1},
 \end{equation}
 where $X^{t}_{\varepsilon_{m}F}$ is the flow of the Hamiltonian, $X_{\varepsilon_{m}F}$ is the vector field
 of the Hamiltonian $\varepsilon_{m}F$ with the symplectic structure $\mi\,  du\wedge d\overline{u}.$
 Let
 \begin{equation}\label{eq37}
H_{m+1}=H_{m}\circ \Psi_{m}.
 \end{equation}
 By \eqref{eq26}, we have
 \begin{equation}\label{eq38}
H_{m}=N_{m}+R_{m},
 \end{equation}
 with
\begin{equation}\label{eq39}
N_{m}=\omega I+\Lambda_{0}^{(m)}u_{0}\overline{u}_{0}+\sum_{j=1}^{\infty}(\Lambda_{j}^{(m)}u_{j})\cdot\overline{u}_{j},
 \end{equation}
 \begin{equation}\label{eq40}
R_{m}=\sum_{l=m}^{\infty}\varepsilon_{l}R_{lm},
 \end{equation}
 \begin{equation}\label{eq41}
R_{lm}=\langle R_{l,m}^{uu}(\theta)u, u\rangle+\langle R_{l,m}^{u\overline{u}}(\theta)u, \overline{u}\rangle+\langle R_{l,m}^{\overline{u}\,\overline{u}}(\theta)\overline{u}, \overline{u}\rangle,
 \end{equation}
 where $(R_{l,m}^{uu}(\theta))^{T}=R_{l,m}^{uu}(\theta), $ $(R_{l,m}^{u\overline{u}}(\theta))^{T}=R_{l,m}^{u\overline{u}}(\theta), $ $(R_{l,m}^{\overline{u} \overline{u}}(\theta))^{T}=R_{l,m}^{\overline{u}\overline{u}}(\theta).$

Recall that the sequence $z=(z_{j}\in \mathbb{C},\;j\in \mathbb{Z})$ can be rewriten as
$$z=(z_{0}, \,z_{j}, \,z_{-j}:\;j=1, 2,\cdots)= u=(u_{j}: \;j=0, 1, 2,\cdots),$$
where $u_{0}=z_{0}$, $u_{j}=(z_{j}, \,z_{-j})^{T},\;j=1, 2,\cdots.$
Suppose $\{\cdot, \cdot\}$ is the Poisson bracket with respect to $\mi\,  dz\wedge d \overline{z},$
  that is $$\{H(z, \overline{z}), F(z, \overline{z})\}=
  \mi\, \left(\frac{\partial H}{\partial z}\cdot\frac{\partial F}{\partial \overline{z}}
  -\frac{\partial H}{\partial \overline{z}}\cdot\frac{\partial F}{\partial z}\right).$$
Define
$$\frac{\partial H}{\partial u_{0}}=\frac{\partial H}{\partial z_{0}},\;\;\;\;\frac{\partial H}{\partial u_{j}}=(\frac{\partial H}{\partial z_{j}},\;\frac{\partial H}{\partial z_{-j}})^{T},\;\;j=1, 2, \cdots,\;\;\sum_{j=0}^{\infty}\frac{\partial H}{\partial u_{j}}\cdot \frac{\partial F}{\partial \overline{u}_{j}}\triangleq \frac{\partial H}{\partial u}\cdot \frac{\partial F}{\partial \overline{u}}.$$
We can verify that
$$\{H(z, \overline{z}), F(z, \overline{z})\}=\{H(u, \overline{u}), F(u, \overline{u})\}=
  \mi\, \left(\frac{\partial H}{\partial u}\cdot\frac{\partial F}{\partial \overline{u}}
  -\frac{\partial H}{\partial \overline{u}}\cdot\frac{\partial F}{\partial u}\right).$$
So $\{\cdot, \cdot\}$ is also the Poisson bracket with respect to $\mi\,  du\wedge d \overline{u}.$
 By combination of \eqref{eq35}-\eqref{eq41} and Taylor formula, we have
 \begin{eqnarray}\label{eq42}
H_{m+1}&=&H_{m}\circ X^{1}_{\varepsilon_{m}F}
\nonumber\\
&=&N_{m}+\varepsilon_{m}\{N_{m},F\}
+\varepsilon^{2}_{m}\int_{0}^{1}(1-\tau)\{\{N_{m}, F\}, F\}\circ X^{\tau}_{\varepsilon_{m}F}d\tau
+\varepsilon_{m}\omega\cdot\partial_{\theta}F\nonumber\\
&&+\varepsilon_{m}R_{mm}+(\sum_{l=m+1}^{\infty}\varepsilon_{l}R_{lm})\circ X^{1}_{\varepsilon_{m}F}
+\varepsilon_{m}^{2}\int_{0}^{1}\{R_{mm}, F\}\circ X^{\tau}_{\varepsilon_{m} F}d\tau.
 \end{eqnarray}
 Let $\Gamma_{K_{m}}$ be a truncation operator.
For any
$$
 f(\theta)=\sum_{k\in \mathbb{Z}^{n}}\widehat{f}(k)e^{\mi\, \langle k, \theta\rangle},\;\theta\in\mathbb{T}^{n}.
$$
 Define, for any given $K_{m}>0,$
 $$\Gamma_{K_{m}} f(\theta)=(\Gamma_{K_{m}} f)(\theta)\triangleq \sum_{|k|\leq K_{m}}\widehat{f}(k)e^{\mi\, \langle k, \theta\rangle},$$
 $$(1-\Gamma_{K_{m}}) f(\theta)=((1-\Gamma_{K_{m}}) f)(\theta)\triangleq \sum_{|k|> K_{m}}\widehat{f}(k)e^{\mi\, \langle k, \theta\rangle}.$$
 Then
 $$f(\theta)=\Gamma_{K_{m}} f(\theta)+(1-\Gamma_{K_{m}})f(\theta).$$
 Let
 \begin{equation}\label{eq43}
\{N_{m}, F\}+\Gamma_{K_{m}}R_{mm}=\langle[R^{u\overline{u}}_{mm}]u, \overline{u}\rangle,
  \end{equation}
  where
 \begin{equation}\label{eq44}
  [R^{u\overline{u}}_{mm}]:=diag \left(\widehat{R}^{u\overline{u}}_{mmjj}(0): j=0, 1, 2, \cdots\right),
  \end{equation}
  and $R^{u\overline{u}}_{mmij}(\theta)$ is the matrix element of $R^{u\overline{u}}_{m,m}(\theta)$ and $\widehat{R}^{u\overline{u}}_{mmij}(k)$ is the
  $k$-Fourier coefficient of $R^{u\overline{u}}_{mmij}(\theta).$
  Then
  \begin{equation}\label{eq45}
  H_{m+1}=N_{m+1}+C_{m+1}R_{m+1},
  \end{equation}
  where
  \begin{equation}\label{eq46}
 N_{m+1}=N_{m}+\varepsilon_{m}\langle [R^{u\overline{u}}_{mm}]u, \overline{u}\rangle
 =\Lambda_{0}^{(m+1)}u_{0}\overline{u}_{0}+\sum_{j=1}^{\infty}( \Lambda_{j}^{(m+1)}u_{j})\cdot\overline{u}_{j},
  \end{equation}
\begin{equation}\label{eq5.23}
\Lambda_{j}^{(m+1)}=\Lambda_{j}^{(m)}+\varepsilon_{m}\widehat{R}_{mmjj}^{u\overline{u}}(0)
=\Lambda_{j}^{(m)}+\sum_{l=0}^{m}\varepsilon_{l}\mu_{j}^{(l)},
\;\mu_{j}^{(m)}:=\widehat{R}_{mmjj}^{u\overline{u}}(0),
  \end{equation}
  \begin{eqnarray}
 C_{m+1}R_{m+1}&=&\varepsilon_{m}(1-\Gamma_{K_{m}})R_{mm}\label{eq5.24}\\
 &+&\varepsilon_{m}^{2}\int_{0}^{1}(1-\tau)\{\{N_{m}, F\},F\}\circ X_{\varepsilon_{m}F}^{\tau}d\tau\label{eq5.27}\\
 &+&\varepsilon_{m}^{2}\int_{0}^{1}\{R_{mm}, F\}\circ X^{\tau}_{\varepsilon_{m} F}d\tau\label{eq5.28}\\
  &+&\left(\sum_{l=m+1}^{\infty}\varepsilon_{l}R_{lm}\right)\circ X_{\varepsilon_{m}F}^{1}.\label{eq5.25}
  \end{eqnarray}
  The equation \eqref{eq43} is called the homological equation. Developing the Poisson bracket $\{N_{m},F\}$ and comparing
  the coefficients of $u_{i}{u}_{j}, u_{i}\overline{u}_{j}, \overline{u}_{i}\overline{u}_{j} (i, j=0, 1, 2, \cdots),$ we get

    \begin{eqnarray}
\omega\cdot\partial_{\theta} F^{uu}(\theta, \omega)
      +\mi\, (\Lambda^{(m)}F^{uu}(\theta, \omega)+F^{uu}(\theta, \omega)\Lambda^{(m)})=\Gamma_{K_{m}}R_{mm}^{uu}(\theta),\label{eq48-1}\\
\omega\cdot\partial _{\theta} F^{\overline{u}\,\overline{u}}(\theta, \omega)-\mi\,
     (\Lambda^{(m)}F^{\overline{u}\,\overline{u}}(\theta, \omega)+F^{\overline{u}\,\overline{u}}(\theta, \omega)\Lambda^{(m)})
=\Gamma_{K_{m}}R_{mm}^{\overline{u}\,\overline{u}}(\theta), \label{eq48-2} \\
\omega\cdot\partial_{\theta} F^{u\overline{u}}(\theta, \omega)
       +\mi\, (F^{u\overline{u}}(\theta, \omega)\Lambda^{(m)}-\Lambda^{(m)}F^{u\overline{u}}(\theta, \omega))
=\Gamma_{K_{m}}R_{mm}^{u\overline{u}}(\theta)-[R_{mm}],\label{eq48-3}
    \end{eqnarray}
where
\begin{equation}\label{eq49}
\Lambda^{(m)}=diag(\Lambda_{j}^{(m)}: j=0, 1, 2, \cdots),
\end{equation}
and we assume
$$\Gamma_{K_{m}}F^{uu}(\theta, \omega)=F^{uu}(\theta, \omega),\; \Gamma_{K_{m}}F^{u\overline{u}}(\theta, \omega)=F^{u\overline{u}}(\theta, \omega),\; \Gamma_{K_{m}}F^{\overline{u}\, \overline{u}}(\theta, \omega)=F^{\overline{u}\, \overline{u}}(\theta, \omega).$$
 $F^{uu}_{ij}(\theta), F^{u\overline{u}}_{ij}(\theta), F^{\overline{u} \,\overline{u}}_{ij}(\theta)$ are written as the matrix elements of $F^{uu}(\theta, \omega), F^{u\overline{u}}(\theta, \omega), F^{\overline{u} \,\overline{u}}(\theta, \omega),$ respectively. More exactly, for $p, q\in\{u, \overline{u}\},$
$$F^{pq}_{ij}(\theta)=\left\{
                      \begin{array}{ll}
                        a_{0,0}(\theta), & {i=j=0;} \\
                        (a_{0,j}(\theta),\;a_{0,-j}(\theta)), & {i=0, \;j=1, 2, \cdots;} \\
                        (a_{i,0}(\theta),\;a_{-i,0}(\theta))^{T}, & {j=0, \;i=1, 2, \cdots;} \\
                        \left(
                          \begin{array}{cc}
                            a_{i,j}(\theta) & a_{i,-j}(\theta) \\
                            a_{-i,j}(\theta) & a_{-i,-j}(\theta) \\
                          \end{array}
                        \right), & {i, j=1, 2, \cdots,}
                      \end{array}
                    \right.$$
where $a_{i,j}(\theta): \mathbb{T}_{s_{m}}^{n}\rightarrow \mathbb{R},\;i, j=0, 1, 2, \cdots.$
Then \eqref{eq48-1}-\eqref{eq48-3} can be rewritten as:
\begin{equation}\label{eq5.31}
\omega\cdot\partial_{\theta} F_{ij}^{uu}(\theta)
       +\mi\, (\Lambda_{i}^{(m)}F_{ij}^{uu}(\theta)+F_{ij}^{uu}(\theta))\Lambda_{j}^{(m)}=
       \Gamma_{K_{m}}R_{mmij}^{uu}(\theta),
\end{equation}
\begin{equation}\label{eq5.32}
     \omega\cdot\partial _{\theta} F_{ij}^{\overline{u}\,\overline{u}}(\theta)-\mi\, (\Lambda_{i}^{(m)}F_{ij}^{\overline{u}\,\overline{u}}(\theta)
     +F_{ij}^{\overline{u}\,\overline{u}}(\theta)\Lambda_{j}^{(m)})
=\Gamma_{K_{m}}R_{mmij}^{\overline{u}\,\overline{u}}(\theta),
\end{equation}
\begin{equation}\label{eq5.33}
\omega\cdot\partial_{\theta} F_{ij}^{u\overline{u}}(\theta)
       -\mi\, (\Lambda_{i}^{(m)}F_{ij}^{u\overline{u}}(\theta)-F_{ij}^{u\overline{u}}(\theta)\Lambda_{j}^{(m)})=
       \Gamma_{K_{m}}R_{mmij}^{u\overline{u}}(\theta),\;i\neq j,
\end{equation}
 \begin{equation}\label{eq5.34}
\omega\cdot\partial_{\theta} F_{ii}^{u\overline{u}}(\theta)
 -\mi\, (\Lambda_{i}^{(m)}F_{ii}^{u\overline{u}}(\theta)-F_{ii}^{u\overline{u}}(\theta)\Lambda_{i}^{(m)})=\Gamma_{K_{m}}R_{mmii}^{u\overline{u}}(\theta)-\widehat{R}_{mmii}(0),
\end{equation}

where $i,j=0, 1, 2, \cdots.$
\section{Solutions of the homological equations}
\begin{lemma}\label{lem7.1}
There exists a compact subset $\Pi_{m+1}^{+-}\subset \Pi_{m}$ with
\begin{equation}\label{eq7.1}
mes(\Pi_{m+1}^{+-})\geq mes \Pi_{m}-C\gamma_{m}^{1/3}
\end{equation}
such that for any $\omega\in \Pi_{m+1}^{+-}$,
the equation \eqref{eq48-3} has a unique solution $F^{u\overline{u}}(\theta, \omega),$
which is defined on the domain $\mathbb{T}_{s_{m+1}}^{n}\times \Pi_{m+1}^{+-},$
with
\begin{equation}\label{eq7.2}
\|JF^{u\overline{u}}(\theta,\omega)J\|_{\mathbb{T}_{s_{m+1}}^{n}\times
\Pi_{m+1}^{+-}}\leq C(m+1)\varepsilon_{m}^{-\frac{2(3n+4)}{N}},
\end{equation}
\begin{equation}\label{eq7.3}
\|JF^{u\overline{u}}(\theta,\omega)J\|^{\mathcal{L}}_{\mathbb{T}_{s_{m+1}}^{n}\times
\Pi_{m+1}^{+-}}\leq C(m+1)\varepsilon_{m}^{-\frac{6(3n+4)}{N}}.
\end{equation}
\end{lemma}
\begin{proof}
By passing to Fourier coefficients, \eqref{eq5.33} can be rewritten as
\begin{equation}\label{eq7.4}
-\langle k, \omega\rangle\widehat{F}_{ij}^{u\overline{u}}(k)+(\Lambda_{i}^{(m)}\widehat{F}_{ij}^{u\overline{u}}(k)-\widehat{F}_{ij}^{u\overline{u}}(k)\Lambda_{j}^{(m)})
=\mi\, \widehat{R}_{mmij}^{u\overline{u}}(k),
\end{equation}
where $i, j=0, 1, 2, \cdots,\, i\neq j,\,k\in \mathbb{Z}^{n}$ with $|k|\leq K_{m}.$
In the following, we always by ``1" denote the identity from some finite dimensional space to itself. By applying $``\Vec"$ to both sides of \eqref{eq7.4}, we have
\begin{equation}\label{eq7.41}
(-\langle k, \omega\rangle(1\otimes 1)+1\otimes \Lambda_{i}^{(m)}-(\Lambda_{j}^{(m)})^{T}\otimes 1) \vec\;\widehat{F}_{ij}^{u\overline{u}}(k)
=\vec\;(\mi\, \widehat{R}_{mmij}^{u\overline{u}}(k)),
\end{equation}
where $A\otimes B$ is the tensor product of $A$ and $B.$
Let $\mu^{ml}_{kij}$ be the $l$-th eigenvalue of $1\otimes \Lambda_{i}^{(m)}-(\Lambda_{j}^{(m)})^{T}\otimes 1,\;l=1, 2, 3, 4.$
Let
$$A_{k}=|k|^{2n+4}+8,$$
and
\begin{equation}\label{eq7.5}
Q_{kijl}^{(m)}\triangleq \left\{\omega\in\Pi_{m}\bigg| \big|-\langle k, \omega\rangle
+\mu^{ml}_{kij}\big|<\frac{(|i-j|+1)\gamma_{m}}{A_{k}}\right\},
\end{equation}
where $i, j=0, 1, 2, \cdots,$ $l=1, 2, 3, 4,$ $k\in\mathbb{Z}^{n}$ with $|k|\leq K_{m},$ and $k\neq 0$ when $i=j.$
Let
\begin{equation*}
\Pi_{m+1}^{+-}=\Pi_{m}\diagdown\bigcup_{|k|\leq K_{m}}\bigcup_{i=1}^{\infty}\bigcup_{j=1}^{\infty}\bigcup_{l=1}^{4}Q_{kijl}^{(m)}.
\end{equation*}
Then for any $\omega\in\Pi_{m+1}^{+-},$ we have
\begin{equation}\label{eq7.6}
\big|-\langle k, \omega\rangle
+\mu^{ml}_{kij}\big|\geq\frac{(|i-j|+1)\gamma_{m}}{A_{k}}.
\end{equation}
Then
\begin{equation}\label{eq7.411}
||(-\langle k, \omega\rangle(1\otimes 1)+1\otimes \Lambda_{i}^{(m)}-(\Lambda_{j}^{(m)})^{T}\otimes 1)^{-1}||_{2}\leq \frac{A_{k}}{(|i-j|+1)\gamma_{m}}.
\end{equation}
Here $||\cdot ||_{2}$ denotes the spectral norm of matrices.
Recall that $R_{mm}^{{u}\overline{u}}(\theta)$ is analytic in the domain $\mathbb{T}_{s_{m}}^{n}$
for any $\omega\in \Pi_{m},$
\begin{equation}\label{eq7.8}
||\widehat{R}_{mmij}^{u\overline{{u}}}(k)||_{2}
\leq \frac{C(m)}{\sqrt {ij}}e^{-s_{m}|k|},
\end{equation}
which implies that $$||\vec\;(\mi\, \widehat{R}_{mmij}^{u\overline{{u}}}(k))||_{2}
\leq \frac{C(m)}{\sqrt {ij}}e^{-s_{m}|k|}.$$
By \eqref{eq7.41}, we have
\begin{eqnarray*}
||\vec\; \widehat{F}_{ij}^{u\overline{u}}(k)||_{2}
\leq \frac{A_{k}}{(|i-j|+1)\gamma_{m}}||\vec(\mi\, \widehat{R}_{mmij}^{{u}\overline{u}}(k))||_{2}
\leq \frac{A_{k}}{\gamma_{m}(|i-j|+1)}\frac{C(m)e^{-s_{m}|k|}}{\sqrt {ij}}.\end{eqnarray*}
Then
\begin{eqnarray}\label{eq7.7}
||\widehat{F}_{ij}^{u\overline{u}}(k)||_{2}
\leq
\frac{(|k|^{2n+4}+8)}{\gamma_{m}(|i-j|+1)}\frac{C(m)e^{-s_{m}|k|}}{\sqrt {ij}},\;i\neq j.
\end{eqnarray}

Now we need the following lemmas:
\begin{lemma}\label{lem7.2}\cite{Bogoljubov1969}
For $0<\delta<1, \nu>1,$ one has
$$\sum_{k\in \mathbb{Z}^{n}}e^{-2|k|\delta}|k|^{\nu}<\left(\frac{\nu}{e}\right)^{\nu}\frac{(1+e)^{n}}{\delta^{\nu+n}}
.$$
\end{lemma}

\begin{lemma}\label{lem7.3}
If $A=(A_{ij})$ is a bounded linear operator on $h_{N},$ then also $B=(B_{ij}: i, j=0, 1, 2,\cdots)$ with
$$||B_{ij}||_{2}\leq\frac{|A_{ij}|}{|i-j|},\;i, j=0, 1, 2,\cdots, \;i\neq j,\\$$
and $\|B\|\leq C\|A\|,$
where $\|\cdot\|$ is $h_{N}\rightarrow h_{N}$ operator norm,
$$B_{ij}=\left\{
                      \begin{array}{ll}
                        b_{0,0}, & {i=j=0;} \\
                        (b_{0,j},\;b_{0,-j}), & {i=0, \;j=1, 2, \cdots;} \\
                        (b_{i,0},\;b_{-i,0})^{T}, & {j=0, \;i=1, 2, \cdots;} \\
                        \left(
                          \begin{array}{cc}
                            b_{i,j} & b_{i,-j} \\
                            b_{-i,j} & b_{-i,-j} \\
                          \end{array}
                        \right), & {i, j=1, 2, \cdots,}
                      \end{array}
                    \right.$$
with $b_{i,j}\in\mathbb{R},\;i, j=0, 1, 2, \cdots.$
\end{lemma}

The proof of this result is similar to the proof of Theorem A.1 of \cite{Poschel1996} and so is omitted. See \cite{Poschel1996} for
the details.

Therefore, by \eqref{eq7.7}, we have
\begin{eqnarray*}
&&\sup_{\theta\in\mathbb{T}^{n}_{s'_{m}}\times \Pi_{m+1}}(||J_{i}F^{u\overline{u}}_{ij}(\theta, \omega)J_{j}||_{2})\\
&\leq &\left(\sum_{|k|\leq K_{m}}(|k|^{2n+4}+8)e^{-(s_{m}-s'_{m})|k|}\right)\frac{C(m)}{\gamma_{m}(|i-j|+1)}\\
&\leq & C\left(\frac{2n+4}{e}\right)^{2n+4}\!(1+e)^{n}\left(\frac{2}{s_{m}-s'_{m}}\right)^{3n+4}
\frac{C(m)}{\gamma_{m}(|i-j|+1)}
\;( \mbox{by Lemma \ref{lem7.2}})\\
&\leq & C\frac{C(m)}{(s_{m}-s'_{m})^{3n+4}}\frac{1}{\gamma_{m}(|i-j|+1)}\\
&\leq & C\varepsilon_{m}^{-\frac{2(3n+4)}{N}}\frac{C(m)}{\gamma_{m}(|i-j|+1)},
\end{eqnarray*}
where $C$ is a constant depending on $n,$ $s'_{m}=s_{m}-\frac{s_{m}-s_{m+1}}{4}.$
By Lemma \ref{lem7.3}, we have
\begin{equation}\label{eq7.9}
\|JF^{u\overline{u}}(\theta, \omega)J\|_{\mathbb{T}^{n}_{s'_{m}}\times \Pi^{+-}_{m+1}}
\leq C\;C(m)\gamma_{m}^{-1}\varepsilon_{m}^{-\frac{2(3n+4)}{N}}
\leq C(m+1)\varepsilon_{m}^{-\frac{2(3n+4)}{N}}.
\end{equation}
It follows $s'_{m}>s_{m+1}$ that
$$\|JF^{u\overline{u}}(\theta, \omega)J\|_{\mathbb{T}^{n}_{s_{m+1}}\times \Pi^{+-}_{m+1}}
\leq\|JF^{u\overline{u}}(\theta, \omega)J\|_{\mathbb{T}^{n}_{s'_{m}}\times \Pi^{+-}_{m+1}}
\leq C(m+1)\varepsilon_{m}^{-\frac{2(3n+4)}{N}}.$$
Applying $\partial_{\omega_{l}}\;(l=1, 2, \cdots, n)$ to both sides of \eqref{eq7.4}, we have
\begin{equation}\label{eq7.10}
-\langle k, \omega\rangle\partial_{\omega_{l}}\widehat{F}_{ij}^{u\overline{u}}(k)
+(\Lambda_{i}^{(m)}\partial_{\omega_{l}}\widehat{F}_{ij}^{u\overline{u}}(k)-\partial_{\omega_{l}}\widehat{F}_{ij}^{u\overline{u}}(k)\Lambda_{j}^{(m)})
=\mi\, \partial_{\omega_{l}}\widehat{R}_{mmij}^{u\overline{u}}(k)+(*),
\end{equation}
where
\begin{equation}\label{eq7.11}
(*)=k_{l}\widehat{F}_{ij}^{u\overline{u}}(k)-\partial_{\omega_{l}}\Lambda_{i}^{(m)}\widehat{F}_{ij}^{u\overline{u}}(k)
+\widehat{F}_{ij}^{u\overline{u}}(k)\partial_{\omega_{l}}\Lambda_{j}^{(m)}.
\end{equation}
By applying $``\Vec"$ to both sides of \eqref{eq7.10}, we have
\begin{equation}\label{eq7.41*}
(-\langle k, \omega\rangle(1\otimes 1)+1\otimes \Lambda_{i}^{(m)}-(\Lambda_{j}^{(m)})^{T}\otimes 1) \vec\;\partial_{\omega_{l}}\widehat{F}_{ij}^{u\overline{u}}(k)
=\vec\;(\mi\, \partial_{\omega_{l}}\widehat{R}_{mmij}^{u\overline{u}}(k)+(*)),
\end{equation}
Recalling $|k|\leq K_{m}=100s_{m}^{-1}2^{m}|\log \varepsilon|,$ and using \eqref{eq28}-\eqref{eq029} with $\nu=m,$
using \eqref{eq7.11}, we have, on $\omega\in\Pi_{m+1},$
\begin{equation}\label{eq7.12}
||(*)||_{2}\leq CK_{m}||\widehat{F}^{u\overline{u}}_{ij}(k)||_{2}.
\end{equation}
According to \eqref{eq31},
\begin{equation}\label{eq7.13}
||  \partial _{\omega_{l}}\widehat{R}^{u\overline{u}}_{mmij}(k)||_{2} \leq \frac{C(m)e^{-s'_{m}|k|}}{\sqrt {ij}}.
\end{equation}
By \eqref{eq7.7}, \eqref{eq7.41*}, \eqref{eq7.12} and \eqref{eq7.13}, we have
\begin{equation}\label{eq7.14}
|| J_{i} \partial _{\omega}\widehat{F}^{u\overline{u}}_{ij}(k)J_{j}||_{2}
\leq \frac{A^{2}_{k}CK_{m}C(m)e^{-s'_{m}|k|}}{\gamma^{2}_{m}(|i-j|+1)}\;\;\mbox{for}\;\; i\neq j.
\end{equation}
Note that $s_{m}>s'_{m}>s_{m+1}.$ Again using Lemma \ref{lem7.2} and Lemma\ref{lem7.3}, we have
\begin{equation}\label{eq7.15}
\|JF^{u\overline{u}}(\theta, \omega)J\|^{\mathcal{L}}_{\mathbb{T}_{s_{m+1}}\times \Pi_{m+1}^{+-}}
=\| J \partial _{\omega}{F}^{u\overline{u}}(\theta, \omega)J\|_{\mathbb{T}_{s_{m+1}}\times \Pi_{m+1}^{+-}}
\leq C(m+1)\varepsilon_{m}^{-\frac{6(3n+4)}{N}}.
\end{equation}
The proof of the measure estimate \eqref{eq7.1} will be postponed to Section 10.
This completes the proof of Lemma \ref{lem7.1}.
\end{proof}

\begin{lemma}\label{lem7.4}
There exists a compact subset $\Pi_{m+1}^{++}\subset \Pi_{m}$ with
\begin{equation}\label{eq a}
mes(\Pi_{m+1}^{++})\geq mes \Pi_{m}-C\gamma_{m}^{1/3}
\end{equation}
such that for any $\omega\in \Pi_{m+1}^{++}$,
the equation \eqref{eq48-1} has a unique solution $F^{uu}(\theta),$
which is defined on the domain $\mathbb{T}_{s_{m+1}}^{n}\times \Pi_{m+1}^{++},$
with
\begin{equation*}
\|JF^{uu}(\theta,\omega)J\|_{\mathbb{T}_{s_{m+1}}^{n}\times
\Pi_{m+1}^{++}}\leq C(m+1)\varepsilon_{m}^{-\frac{2(3n+4)}{N}},
\end{equation*}
\begin{equation*}
\|JF^{uu}(\theta,\omega)J\|^{\mathcal{L}}_{\mathbb{T}_{s_{m+1}}^{n}\times
\Pi_{m+1}^{++}}\leq C(m+1)\varepsilon_{m}^{-\frac{6(3n+4)}{N}}.
\end{equation*}
\end{lemma}

\begin{lemma}\label{lem7.5}
There exists a compact subset $\Pi_{m+1}^{--}\subset \Pi_{m}$ with
\begin{equation}\label{eq b}
mes(\Pi_{m+1}^{--})\geq mes \Pi_{m}-C\gamma_{m}^{1/3}
\end{equation}
such that for any $\omega\in \Pi_{m+1}^{--}$,
the equation \eqref{eq48-2} has a unique solution $F^{\overline{u}\overline{u}}(\theta),$
which is defined on the domain $\mathbb{T}_{s_{m+1}}^{n}\times \Pi_{m+1}^{--},$
with
\begin{equation*}
\|JF^{\overline{u}\overline{u}}(\theta,\omega)J\|_{\mathbb{T}_{s_{m+1}}^{n}\times
\Pi_{m+1}^{--}}\leq C(m+1)\varepsilon_{m}^{-\frac{2(3n+4)}{N}},
\end{equation*}
\begin{equation*}
\|JF^{\overline{u}\overline{u}}(\theta,\omega)J\|^{\mathcal{L}}_{\mathbb{T}_{s_{m+1}}^{n}\times
\Pi_{m+1}^{--}}\leq C(m+1)\varepsilon_{m}^{-\frac{6(3n+4)}{N}}.
\end{equation*}
\end{lemma}
The proofs of Lemma \ref{lem7.4} and Lemma \ref{lem7.5} are simpler than that of Lemma \ref{lem7.1},
so we omit them.

Let
$$\Pi_{m+1}=\Pi_{m+1}^{+-}\bigcap\Pi_{m+1}^{++}\bigcap\Pi_{m+1}^{--}.$$
By \eqref{eq7.1}, \eqref{eq a} and \eqref{eq b}, we have
$$mes \Pi_{m+1}\geq mes \Pi _{m}-C\gamma_{m}^{1/3}.$$
\section{Coordinate change $\Psi$ by $\varepsilon_{m} F$}
Recall $\Psi=\Psi_{m}=X_{\varepsilon_{m}F}^{t}\big| _{t=1},$ where $X^{t}_{\varepsilon_{m}F}$ is the flow of the Hamiltonian $\varepsilon_{m} F$ and
$X_{\varepsilon_{m}F}$ is the vector field with symplectic $\mi\, du\wedge d\overline{u}.$
So
$$\mi\,  \dot{u}=\varepsilon_{m}\frac{\partial F}{\partial \overline{u}},\;
-\mi\, \dot{\overline{u}}=\varepsilon_{m}\frac{\partial F}{\partial {u}},\;\dot{\theta}=\omega.$$
More exactly,
$$\left\{
    \begin{array}{ll}
      \mi\, \dot{u}=\varepsilon_{m}(F^{u\overline{u}}(\theta, \omega)u
+2F^{\overline{u}\,\overline{u}}(\theta, \omega)\overline{u}),\;\theta=\omega t,\\
     -\mi\, \dot{\overline{u}}=\varepsilon_{m}(2F^{{u}{u}}(\theta, \omega){u}+F^{u\overline{u}}(\theta, \omega)\overline{u}),\;\theta=\omega t, \\
 \dot{\theta}=\omega.
    \end{array}
  \right.
$$

Let $\widetilde{u}=\left(
         \begin{array}{c}
           u \\
           \overline{u} \\
         \end{array}
       \right),$
\begin{equation}\label{eq8.0}
B_{m}=\left(
            \begin{array}{cc}
            -\mi\,  F^{u\overline{u}}(\theta, \omega) & -2\mi\, F^{\overline{u}\,\overline{u}}(\theta, \omega) \\
             2\mi\,  F^{uu}(\theta, \omega) & \mi\,  F^{u\overline{u}}(\theta, \omega)\\
            \end{array}
          \right).\;\;\mbox{Recall }\;\;\theta=\omega t.
\end{equation}
Then
\begin{equation}\label{eq8.1}
\frac{d\widetilde{u}(t)}{dt}=\varepsilon_{m}B_{m}(\theta)\widetilde{u},\;\;\dot{\theta}=\omega.
\end{equation}
Let $\widetilde{u}(0)=\widetilde{u}_{0}\in h_{N}\times h_{N},\,\theta(0)=\theta_{0}\in \mathbb{T}^{n}_{s_{m+1}}$ be
initial value.
Then
\begin{equation}\label{eq8.2}
\left\{
  \begin{array}{ll}
    \widetilde{u}(t)=\widetilde{u}_{0}+\int_{0}^{t}\varepsilon_{m}B_{m}(\theta_{0}+\omega s)\widetilde{u}(s)ds, \\
    \theta(t)=\theta_{0}+\omega t.
  \end{array}
\right.
\end{equation}
By Lemmas \ref{lem7.1}, \ref{lem7.4} and \ref{lem7.5},
\begin{equation}\label{eq8.3}
\|J B_{m}(\theta)J\|_{\mathbb{T}^{n}_{s_{m+1}}\times \Pi_{m+1}}\leq C(m+1)\varepsilon_{m}^{-\frac{2(3n+4)}{N}},
\end{equation}
\begin{equation}\label{eq8.4}
\|J B_{m}(\theta) J\|^{\mathcal{L}}_{\mathbb{T}^{n}_{s_{m+1}}\times \Pi_{m+1}}\leq C(m+1)\varepsilon_{m}^{-\frac{6(3n+4)}{N}}.
\end{equation}
It follows from \eqref{eq8.2} that
$$\widetilde{u}(t)-\widetilde{u}_{0}=\int_{0}^{t}\varepsilon_{m}B_{m}(\theta_{0}+\omega s)\widetilde{u}_{0}ds+\int_{0}^{t}\varepsilon_{m}B_{m}(\theta_{0}+\omega s)(\widetilde{u}(s)-\widetilde{u}_{0})ds.$$
Moreover, for $t\in [0,1]$, $\|\widetilde{u}_{0}\|_{N}\leq 1,$
\begin{equation}\label{eq8.5}
\|\widetilde{u}(t)-\widetilde{u}_{0}\|_{N}\leq \varepsilon_{m}C(m+1)\varepsilon_{m}^{-\frac{2(3n+4)}{N}}
+\int_{0}^{t}\varepsilon_{m}\|B_{m}(\theta_{0}+\omega s)\|\|\widetilde{u}(s)-\widetilde{u}_{0}\|_{N}ds,
\end{equation}
where $\|\cdot\|$ is the operator norm from $h_{N}\times h_{N}\rightarrow h_{N}\times h_{N}.$
By Gronwall's inequality,
\begin{equation}\label{eq8.6}
\|\widetilde{u}(t)-\widetilde{u}_{0}\|_{N}\leq C(m+1)\varepsilon_{m}^{1-\frac{2(3n+4)}{N}}
\exp\left(\int_{0}^{t}\varepsilon_{m}\|B_{m}(\theta_{0}+\omega s)\|ds\right)\leq \varepsilon_{m}^{1/2}.
\end{equation}
Thus,
\begin{equation}\label{eq8.7}
\Psi_{m}: \mathbb{T}^{n}_{s_{m+1}}\times \Pi_{m+1}\rightarrow \mathbb{T}^{n}_{s_{m}}\times \Pi_{m},
\end{equation}
and
\begin{equation}\label{eq8.8}
\|\Psi_{m}-id\|_{h_{N}\to h_{N}}\leq \varepsilon_{m}^{1/2}.
\end{equation}
Since \eqref{eq8.1} is linear, $\Psi_{m}$ is a linear coordinate change. According to \eqref{eq8.2}, construct
Picard sequence:
$$\left\{
    \begin{array}{ll}
      \widetilde{u}_{0}(t)=\widetilde{u}_{0}, \\
      \widetilde{u}_{j+1}(t)=\widetilde{u}_{0}+\int_{0}^{t}\varepsilon_{m}B(\theta_{0}+\omega s)\widetilde{u}_{j}(s)ds,\;j=0, 1, 2,\cdots.
    \end{array}
  \right.
$$
By \eqref{eq8.8}, this sequence with $t=1$ goes to
\begin{equation}\label{eq8.10}
\Psi_{m}(u_{0})=\widetilde{u}(1)=(id+P_{m}(\theta_{0}))u_{0},
\end{equation}
where $id$ is the identity from $h_{N}\times h_{N}\rightarrow h_{N}\times h_{N},$ and $P_{m}(\theta_{0})$ is an operator form $h_{N}\times h_{N}\rightarrow h_{N}\times h_{N}$
for any fixed $\theta_{0}\in \mathbb{T}^{n}_{{s_{m+1}}}, \omega\in\Pi_{m+1},$
and is analytic in $\theta_{0}\in \mathbb{T}^{n}_{{s_{m+1}}},$ with
\begin{equation}\label{eq8.11}
\|P_{m}(\theta_{0})\|_{\mathbb{T}^{n}_{s_{m+1}}\times \Pi_{m+1}}\leq \varepsilon_{m}^{1/2}.
\end{equation}
Note that \eqref{eq8.1} is a Hamiltonian system, so $P_{m}(\theta_{0})$ is a symplectic linear operator from
$h_{N}\times h_{N}$ to $h_{N}\times h_{N}.$
\section{Estimates of remainders}
The aim of this section is devoted to the estimates of the remainders:
$$C_{m+1}R_{m+1}=(6.14)+\cdots + (6.17).$$
\begin{itemize}
  \item Estimate of \eqref{eq5.24}.

By \eqref{eq41}, let
$$\widetilde{R}_{mm}=\widetilde{R}_{mm}(\theta)=\left(
            \begin{array}{cc}
            R_{_{m,m}}^{uu}(\theta) & \frac{1}{2}R_{m,m}^{u\overline{u}}(\theta) \\
              \frac{1}{2}R_{m,m}^{u\overline{u}}(\theta) & R_{m,m}^{\overline{u}\overline{u}}(\theta)\\
            \end{array}
          \right),$$
then $$R_{mm}=\langle \widetilde{R}_{mm}\left(
                                          \begin{array}{cc}
                                            u \\
                                            \overline{u}
                                          \end{array}
                                        \right)
, \left(
 \begin{array}{cc}
  u \\
  \overline{u}
  \end{array}
  \right)\rangle.$$ So
$$(1-\Gamma_{K_{m}})R_{mm}\triangleq
\langle (1-\Gamma_{K_{m}})\widetilde{R}_{mm}
\left(
          \begin{array}{c}
   u \\
\overline{u} \\
     \end{array}
      \right),
\left(
          \begin{array}{c}
   u \\
\overline{u} \\
     \end{array}
      \right)
\rangle .$$
By the definition of truncation operator $\Gamma_{K_{m}},$
$$(1-\Gamma_{K_{m}})\widetilde{R}_{mm}=\sum_{|k|> K_{m}}\widehat{\widetilde{R}}_{mm}(k)e^{\mi\,  \langle k, \theta\rangle},
\;\theta\in \mathbb{T}^{n}_{s_{m}},\;\omega\in \Pi_{m}.$$
Since $\widetilde{R}_{mm}=\widetilde{R}_{mm}(\theta)$ is analytic in $\theta\in \mathbb{T}^{n}_{s_{m}},$
\begin{eqnarray*}
&&\sup_{(\theta, \omega)\in\mathbb{T}^{n}_{s_{m+1}}\times \Pi_{m+1}}\|J(1-\Gamma_{K_{m}})\widetilde{R}_{mm}J\|_{h_{N}\to h_{N}}^{2}
\leq\sum_{|k|> K_{m}}\|J\widehat{\widetilde{R}}_{mm}(k)J\|_{N}^{2}e^{2|k|s_{m+1}}\\
&&\leq\|J\widetilde{R}_{mm}J\|_{\mathbb{T}_{s_{m}}^{n}\times \Pi_{m}}^{2}\sum_{|k|>K_{m}}e^{-2(s_{m}-s_{m+1})|k|}
\\
&&\leq{C^{2}(m)\varepsilon_{m}^{-1}e^{-2K_{m}(s_{m}-s_{m+1})}}\;\mbox{(by \eqref{eq30})}\\
&&\leq C^{2}(m)\varepsilon_{m}^{2},
\end{eqnarray*}
which leads to
$$\|J(1-\Gamma_{K_{m}})\widetilde{R}_{mm}J\|_{\mathbb{T}^{n}_{s_{m+1}}\times \Pi_{m+1}}\leq \varepsilon_{m}C(m+1).$$
Thus,
\begin{equation*}\label{eq9.1}
\|\varepsilon_{m}J(1-\Gamma_{K_{m}})\widetilde{R}_{mm}J\|_{\mathbb{T}^{n}_{s_{m+1}}\times \Pi_{m+1}}
\leq \varepsilon^{2}_{m}C(m+1)\leq \varepsilon_{m+1}C(m+1).
\end{equation*}
Similarly,
\begin{equation*}
\|\varepsilon_{m}J(1-\Gamma_{K_{m}})\widetilde{R}_{mm}J\|^{\mathcal{L}}_{\mathbb{T}^{n}_{s_{m+1}}\times \Pi_{m+1}}
\leq \varepsilon^{2}_{m}C(m+1)\leq \varepsilon_{m+1}C(m+1).
\end{equation*}
  \item Estimate of \eqref{eq5.28}.

Let
\begin{equation*}
S_{m}=\left(
            \begin{array}{cc}
             F^{uu}(\theta, \omega) & \frac{1}{2}F^{u\,\overline{u}}(\theta, \omega) \\
            \frac{1}{2} F^{u\overline u}(\theta, \omega) & F^{\overline {u}\overline{u}}(\theta, \omega)\\
            \end{array}
          \right),
\end{equation*}
Then we have
$$F=\langle S_{m}(\theta)\left(
                       \begin{array}{c}
                         u \\
                         \overline{u} \\
                       \end{array}
                     \right),
\left(
                       \begin{array}{c}
                         u \\
                         \overline{u} \\
                       \end{array}
                     \right)
\rangle =\langle S_{m}\widetilde{u}, \widetilde{u}\rangle,\;\widetilde{u}=\left(
                       \begin{array}{c}
                         u \\
                         \overline{u} \\
                       \end{array}
                     \right).$$
Then
\begin{equation}\label{eq9.2}
\varepsilon_{m}^{2}\{R_{mm}, F\}
=4\varepsilon_{m}^{2} \langle {\widetilde{R}}_{mm}(\theta)\mathcal{J}S_{m}(\theta)\widetilde{u}, \widetilde{u}\rangle.
\end{equation}
Note $\mathbb{T}^{n}_{s_{m}}\times \Pi_{m}\supset \mathbb{T}^{n}_{s_{m+1}}\times \Pi_{m+1}.$ By \eqref{eq30} and \eqref{eq31} with $l=m, v=m,$
\begin{equation}\label{eq9.02}
\|\widetilde{R}_{mm}(\theta)\|_{\mathbb{T}^{n}_{s_{m+1}}\times \Pi_{m+1}}
\leq \|\widetilde{R}_{mm}(\theta)\|_{\mathbb{T}^{n}_{s_{m}}\times \Pi_{m}}\leq C(m),
\end{equation}
\begin{equation}\label{eq9.002}
\|\widetilde{R}_{mm}(\theta)\|^{\mathcal{L}}_{\mathbb{T}^{n}_{s_{m+1}}\times \Pi_{m+1}}\leq C(m).
\end{equation}
Let $\widetilde{S}_{m}(\theta)=\mathcal{J}S_{m}(\theta).$
Then by Lemmas \ref{lem7.1}, \ref{lem7.4} and \ref{lem7.5}, we have
\begin{equation}\label{eq8.04}
\|J \widetilde{S}_{m}(\theta)J\|_{\mathbb{T}^{n}_{s_{m+1}}\times \Pi_{m+1}}\leq C(m+1)\varepsilon_{m}^{-\frac{2(3n+4)}{N}},
\end{equation}
\begin{equation}\label{eq8.05}
\|J \widetilde{S}_{m}(\theta)J\|^{\mathcal{L}}_{\mathbb{T}^{n}_{s_{m+1}}\times \Pi_{m+1}}\leq C(m+1)\varepsilon_{m}^{-\frac{6(3n+4)}{N}},
\end{equation}
and
\begin{equation}\label{eq9.3}
\|\widetilde{R}_{mm}\mathcal{J}S_{m}\|_{\mathbb{T}^{n}_{s_{m+1}}\times \Pi_{m+1}}=\|\widetilde{R}_{mm}\widetilde{S}_{m}\|_{\mathbb{T}^{n}_{s_{m+1}}\times \Pi_{m+1}}\leq C(m)C(m+1)\varepsilon_{m}^{-\frac{2(3n+4)}{N}}.
\end{equation}
Set
\begin{equation*}
[\widetilde{{R}}_{mm}, \widetilde{S}_{m}]=\widetilde{{R}}_{mm}\widetilde{S}_{m}+(\widetilde{{R}}_{mm}\widetilde{S}_{m})^{T}.
\end{equation*}
Note that the vector field is linear. So, by Taylor formula, one has
$$
\eqref{eq5.28}=\varepsilon_{m}^{2}\langle \widetilde{R}^{*}_{m}(\theta)\widetilde{u}, \widetilde{u}\rangle, $$
where
$$\widetilde{R}^{*}_{m}(\theta)=2^{2}\widetilde{R}_{mm}\widetilde{S}_{m}
+\sum_{j=2}^{\infty}\frac{2^{j+1}\varepsilon^{j-1}_{m}}{j!}\underbrace{[\cdots [\widetilde{R}_{mm}, \widetilde{S}_{m}], \cdots , \widetilde{S}_{m}]}_{j-1-\mbox{fold}}\widetilde{S}_{m}.
$$
By \eqref{eq9.02} and \eqref{eq8.04},
\begin{eqnarray*}
\|J\widetilde{R}^{*}_{m}(\theta)J\|_{\mathbb{T}^{n}_{s_{m+1}}\times \Pi_{m+1}}
&\leq &\sum_{j=1}^{\infty}\frac{C(m)C(m+1)\varepsilon_{m}^{j-1}(\varepsilon_{m}^{-\frac{2(3n+4)}{N}})^{j}}{j!}\\
&\leq &C(m)C(m+1)\varepsilon_{m}^{-\frac{2(3n+4)}{N}}.
\end{eqnarray*}
By \eqref{eq9.002} and \eqref{eq8.05},
$$\|J\widetilde{R}^{*}_{m}(\theta)J\|^{\mathcal{L}}_{\mathbb{T}^{n}_{s_{m+1}}\times \Pi_{m+1}}\leq C(m)C(m+1)\varepsilon_{m}^{-\frac{6(3n+4)}{N}}.$$
Thus,
\begin{equation}\label{eq*1}
||\varepsilon_{m}^{2}J\widetilde{R}_{m}^{*}J||_{\mathbb{T}^{n}_{s_{m+1}}\times \Pi_{m+1}}
\leq C(m)C(m+1)\varepsilon_{m}^{2-\frac{2(3n+4)}{N}}\leq C(m+1)\varepsilon_{m+1},
\end{equation}
and
\begin{equation}\label{eq*2}
||\varepsilon_{m}^{2}J\widetilde{R}_{m}^{*}J||^{\mathcal{L}}_{\mathbb{T}^{n}_{s_{m+1}}\times \Pi_{m+1}}
\leq C(m)C(m+1)\varepsilon_{m}^{2-\frac{6(3n+4)}{N}}\leq C(m+1)\varepsilon_{m+1}.
\end{equation}

   \item Estimate of \eqref{eq5.27}

By \eqref{eq43},
\begin{equation*}\label{eq9.7}
\{N_{m}, F\}=\langle [R^{u \overline{u}}_{mm}]u, \overline{u}\rangle-\Gamma_{K_{m}}R_{mm}\triangleq R_{mm}^{*}.
\end{equation*}
Thus,
\begin{equation}\label{eq9.8}
\eqref{eq5.27}=\varepsilon_{m}^{2}\int_{0}^{1}(1-\tau)\{R_{mm}^{*}, F\}\circ X_{\varepsilon_{m}F}^{\tau} d\omega.
\end{equation}
Note $R_{mm}^{*}$ is a quadratic polynomial in $u$ and $\overline{u}.$ So we write
\begin{equation}\label{eq9.9}
R^{*}_{mm}=\langle \mathcal{R}_{m}(\theta, \omega)\widetilde{u}, \widetilde{u}\rangle,\; \widetilde{u}=\left(
                                                                 \begin{array}{c}
                                                                   u \\
                                                                   \overline{u}\\
                                                                 \end{array}
                                                               \right).
\end{equation}
By \eqref{eq29} and \eqref{eq029} with $l=\nu=m,$ and using \eqref{eq8.04} and \eqref{eq8.05},
\begin{equation}\label{eq9.10}
\|J\mathcal{R}_{m}J\|_{\mathbb{T}^{n}_{s_{m}+1}\times \Pi_{m+1}}\leq C(m)\varepsilon_{m}^{-\frac{2(3n+4)}{N}},\;\;
\|J\mathcal{R}_{m}J\|^{\mathcal{L}}_{\mathbb{T}^{n}_{s_{m}+1}\times \Pi_{m+1}}\leq C(m)\varepsilon_{m}^{-\frac{6(3n+4)}{N}},
\end{equation}
where $\|\cdot\|$ is the operator norm in $h_{N}\times h_{N}\rightarrow h_{N}\times h_{N}.$
Recall $F=\langle S_{m}(\theta, \omega)\widetilde{u}, \widetilde{u}\rangle.$
Set
\begin{equation}\label{eq9.11}
[\mathcal{R}_{m}, \widetilde{S}_{m}]=\mathcal{R}_{m}\widetilde{S}_{m}+(\mathcal{R}_{m}\widetilde{S}_{m})^{T}
.
\end{equation}
Using Taylor formula to \eqref{eq9.8}, we get
\begin{eqnarray*}
\eqref{eq5.27}&=& \frac{\varepsilon_{m}^{2}}{2!}\{R^{*}_{mm}, F\}+\cdots +\frac{\varepsilon_{m}^{j}}{j!}
\underbrace{\{\cdots\{R^{*}_{mm}, F\},\cdots, F\}}_{j-\mbox{fold}}+\cdots\\
&=&\Bigg\langle\left(\sum_{j=2}^{\infty}\frac{2^{j}\varepsilon^{j}_{m}}{j!}\underbrace{[\cdots [\mathcal{R}_{m}, \widetilde{S}_{m}], \cdots , \widetilde{S}_{m}]}_{j-1-\mbox{fold}} \widetilde{S}_{m}\right)\widetilde{u}, \widetilde{u}\Bigg\rangle\\
&\triangleq &\langle \mathcal{R}^{**}(\theta, \omega)\widetilde{u},\widetilde{u}\rangle.
\end{eqnarray*}
By \eqref{eq8.04},\eqref{eq9.10} and \eqref{eq9.11}, we have
\begin{eqnarray}\label{eq9.12}
&&\|J\mathcal{R}^{**}(\theta, \omega)J\|_{\mathbb{T}^{n}_{s_{m+1}}\times \Pi_{m+1}}\nonumber\\
&\leq & \sum_{j=2}^{\infty}\frac{2^{j+1}}{j!}\|J\mathcal{R}_{m}(\theta, \omega)J\|_{\mathbb{T}^{n}_{s_{m}}\times \Pi_{m}}
(\|J\widetilde{S}_{m}J\|_{\mathbb{T}^{n}_{s_{m+1}}\times \Pi_{m+1}}\varepsilon_{m})^{j}\nonumber\\
&\leq & \sum_{j=2}^{\infty}\frac{C(m)}{j!}\left(\varepsilon_{m}C(m+1)\varepsilon_{m}^{-\frac{2(3n+4)}{N}}\right)^{j}\nonumber\\
&\leq & C(m+1)\varepsilon_{m}^{4/3}=C(m+1)\varepsilon_{m+1}.
\end{eqnarray}
Similarly,
\begin{equation}\label{eq9.13}
\|J\mathcal{R}^{**}(\theta, \omega)J\|^{\mathcal{L}}_{\mathbb{T}^{n}_{s_{m+1}}\times \Pi_{m+1}}\leq C(m+1)\varepsilon_{m+1}.
\end{equation}
  \item Estimate of \eqref{eq5.25}
\begin{eqnarray}\label{eq9.19}
\eqref{eq5.25}=\sum_{l=m+1}^{\infty}\varepsilon_{l}(R_{lm}\circ X_{\varepsilon_{m}F}^{1}).
\end{eqnarray}
We can write
$$R_{lm}=\langle \widetilde{R}_{lm}(\theta)\widetilde{u}, \widetilde{u}\rangle.$$
Then, by Taylor formula, one has
$$R_{lm}\circ X_{\varepsilon_{m}F}^{1}=R_{lm}+\sum_{j=1}^{\infty}\frac{1}{j!}\langle \widetilde{R}_{lmj} \widetilde{u}, \widetilde{u}\rangle,$$
where
$$\widetilde{R}_{lmj}=2^{j+1}\underbrace{[\cdots [\widetilde{{R}}_{lm}, \widetilde{S}_{m}], \cdots]}_{j-1-\mbox{fold}}\widetilde{S}_{m}\varepsilon_{m}^{j}.$$
By \eqref{eq30}, \eqref{eq31},
$$\|J\widetilde{R}_{lm}J\|_{\mathbb{T}^{n}_{s_{l}}\times \Pi_{m}}\leq C(l),\;\;\|J\widetilde{R}_{lm}J\|^{\mathcal{L}}_{\mathbb{T}^{n}_{s_{l}}\times \Pi_{m}}\leq C(l).$$
Combining the last inequalities with \eqref{eq8.04} and \eqref{eq8.05}, one has
\begin{eqnarray*}
&&\|J\widetilde{R}_{lmj}J\|_{\mathbb{T}^{n}_{s_{l}}\times \Pi_{m+1}}\\
&\leq & \|J\widetilde{R}_{lm}J\|_{\mathbb{T}^{n}_{s_{l}}\times \Pi_{m+1}} (||J\widetilde{S}_{m}J||_{\mathbb{T}^{n}_{s_{l}}\times \Pi_{m+1}}4\varepsilon_{m})^{j}\\
&\leq & C^{2}(m)(\varepsilon_{m}\varepsilon_{m}^{-\frac{2(3n+4)}{N}})^{j},
\end{eqnarray*}
where $||J^{-1}||_{\mathbb{T}^{n}_{sl}\times \Pi_{m+1}}\leq C$ is used,
and
\begin{eqnarray*}
&&||J\widetilde{R}_{lmj}J||^{\mathcal{L}}_{\mathbb{T}^{n}_{s_{l}}\times \Pi_{m+1}}\\
&\leq & ||J\widetilde{R}_{lm}J||^{\mathcal{L}}_{\mathbb{T}^{n}_{s_{l}}\times \Pi_{m+1}}(||J\widetilde{S}_{m}J||_{\mathbb{T}^{n}_{s_{l}}\times \Pi_{m+1}}4\varepsilon_{m})^{j}\\
&&+
||J\widetilde{R}_{lm}J||_{\mathbb{T}^{n}_{s_{l}}\times \Pi_{m+1}}(||J\widetilde{S}_{m}J||^{\mathcal{L}}_{\mathbb{T}^{n}_{s_{l}}\times \Pi_{m+1}}\varepsilon_{m})^{j}\\
&\leq &C^{2}(m)(\varepsilon_{m}\varepsilon_{m}^{-\frac{6(3n+4)}{N}})^{j}.
\end{eqnarray*}
Thus, let
$$\overline{R}_{l,m+1}:=\widetilde{R}_{lm}+\sum_{j=1}^{\infty}\frac{1}{j!}\widetilde{R}_{lmj},$$
then
\begin{equation}\label{eq*3}
\eqref{eq5.25}=\sum_{l=m+1}^{\infty}\varepsilon_{l}\langle \overline{R}_{l,m+1}\widetilde{u}, \widetilde{u}\rangle
\end{equation}
and
\begin{equation}\label{eq*4}
||J\overline{R}_{l,m+1}J||_{\mathbb{T}^{n}_{s_{l}}\times \Pi_{m+1}}\leq C^{2}(m)\leq C(m+1),\;\;
||J\overline{R}_{l,m+1}J||^{\mathcal{L}}_{\mathbb{T}^{n}_{s_{l}}\times \Pi_{m+1}}\leq C^{2}(m)\leq C(m+1).
\end{equation}
As a whole, the remainder $R_{m+1}$ can be written  as
$$C_{m+1}R_{m+1}=\sum_{l=m+1}^{\infty}\varepsilon_{l}(\langle R^{uu}_{l, \nu}(\theta)u,u\rangle
+\langle R^{u\overline{u}}_{l, \nu}(\theta)u,\overline{u}\rangle)+\langle R^{\overline{u}\overline{u}}_{l, \nu}(\theta)\overline{u},\overline{u}\rangle),\;\;\nu=m+1,$$
where, for $p,q\in \{u,\overline{u}\},$ $R_{l, \nu}^{p,q}$ satisfies \eqref{eq30} and \eqref{eq31} with
$\nu=m+1, l\geq m+1.$
This shows that Assumption $(A2)_{\nu}$ with $\nu=m+1$ holds true.

By \eqref{eq5.23}, we know
$$\mu_{j}^{(m)}=\widehat{R}_{mmjj}^{u\overline{u}}(0).$$
Taking $p=u, q=\overline{u}$ into \eqref{eq30} and \eqref{eq31}, we have
\begin{eqnarray*}
&&|\mu_{j}^{(m)}|_{\Pi_{m}}\leq
|R^{u\overline{u}}_{mmjj}(\theta, \omega)|/j\leq C(m)/j,\\
&&|\mu_{j}^{(m)}|_{\Pi_{m}}^{\mathcal{L}}\leq |\partial_{\omega}R^{u\overline{u}}_{mmjj}(\theta, \omega)|/j\leq C(m)/j.
\end{eqnarray*}
This shows that Assumption $(A1)_{\nu}$ with $\nu=m+1$ holds true.
\end{itemize}
\section{Estimate of measure}
In this section, $C$ denotes a universal constant, which may be different in different places.
\begin{lemma}\label{lem10.2}
If $|i|, |j|>>1,$ then
\begin{equation}\label{eq10.001}
\mu^{ml}_{ijk}=\rho\sqrt{\lambda_{i}}-\rho\sqrt{\lambda_{j}}+O(\frac{\varepsilon_{0}}{|i|})+O(\frac{\varepsilon_{0}}{|j|}),
\end{equation}
 where $\lambda_{k}={k^{2}+M},\;k\in\mathbb{Z},$ $\mu^{ml}_{kij}$ is the $l$-th eigenvalue of $1\otimes \Lambda_{i}^{(m)}-(\Lambda_{j}^{(m)})^{T}\otimes 1,\;i, j= 1, 2, \cdots,\;i\neq j,\;l=1, 2, 3, 4$ ( for more details, see Section 7, the proof of Lemma \ref{lem7.1}.)
\end{lemma}
\begin{proof}
Recall that $$\Lambda_{i}^{(m)}=\rho\sqrt{\lambda_{i}}E_{22}+ O(\frac{\varepsilon_{0}}{|i|}),\;i\neq 0.$$
By computation, we have
\begin{eqnarray}
1\otimes \Lambda_{i}^{(m)}-(\Lambda_{j}^{(m)})^{T}\otimes 1&=&\rho\sqrt{\lambda_{i}}(E_{22}\otimes E_{22})-\rho\sqrt{\lambda_{j}}(E_{22}\otimes E_{22})+E_{22}\otimes G_{i}+G_{j}\otimes E_{22}\nonumber\\
&=& \rho(\sqrt{\lambda_{i}}-\sqrt{\lambda_{j}})E_{44}+E_{22}\otimes G_{i}+G_{j}\otimes E_{22},
\end{eqnarray}
where $G_{i}$ is a $2\times 2$ matrix such that $|G_{i}|\leq \frac{C\varepsilon_{0}}{|i|}.$
Then
$$|1\otimes \Lambda_{i}^{(m)}-(\Lambda_{j}^{(m)})^{T}\otimes 1-\rho(\sqrt{\lambda_{i}}-\sqrt{\lambda_{j}})E_{44}|\leq (\frac{C}{|i|}+\frac{C}{|j|})\varepsilon_{0}.$$
Note that $1\otimes \Lambda_{i}^{(m)}-(\Lambda_{j}^{(m)})^{T}\otimes 1$ is Hermitian. By the perturbation theory for eigenvalue of matrices, we obtain \eqref{eq10.001}.
 \end{proof}
Now let us return to \eqref{eq7.5}
\begin{equation}\label{eq10.01}
Q_{kijl}^{(m)}\triangleq \left\{\omega\in\Pi_{m}\bigg| \big|-\langle k, \omega\rangle
+\mu^{ml}_{kij}\big|<\frac{(|i-j|+1)\gamma_{m}}{A_{k}}\right\},\;\;A_{k}=|k|^{2n+4}+8.
\end{equation}
{\bf Case 1.} $i\neq j$.
If $ Q_{kijl}^{(m)}=\varnothing,$ then $mes Q_{kijl}^{(m)}=0.$ So we assume $ Q_{kijl}^{(m)}\neq\varnothing.$
Then there exists $\omega\in \Pi_{m}$ such that
\begin{equation}\label{eq10.04}
|-\langle k, \omega\rangle+\mu^{ml}_{kij}|<\frac{|i-j|+1}{A_{k}}\gamma_{m}.
\end{equation}
\begin{itemize}
  \item [(1.1)]$k\neq 0.$
\\
By Lemma \ref{lem10.2},
\begin{equation}\label{eq10.06}
|\mu^{ml}_{kij}|=|\rho\sqrt{\lambda_{i}}-\rho\sqrt{\lambda_{j}}+O(\frac{\varepsilon_{0}}{|i|})+O(\frac{\varepsilon_{0}}{|j|})|\geq \frac{1}{2}|\sqrt{\lambda_{i}}-\sqrt{\lambda_{j}}|.
\end{equation}
Furthermore, it is easy to verify that
\begin{equation}\label{eq10.18}
|\sqrt{\lambda_{i}}-\sqrt{\lambda_{j}}|\geq\frac{4(|i-j|+1)\gamma_{m}}{A_{k}}.
\end{equation}
Then by \eqref{eq10.04}, \eqref{eq10.06} and \eqref{eq10.18}, one has
\begin{eqnarray*}
|\langle k, \omega\rangle |&\geq &|\mu^{ml}_{kij}|-\frac{(|i-j|+1)\gamma_{m}}{A_{k}}
\geq  \frac{1}{2}|\sqrt{\lambda_{i}}-\sqrt{\lambda_{j}}|-\frac{(|i-j|+1)\gamma_{m}}{A_{k}}\\
&\geq & \frac{1}{4}|\sqrt{\lambda_{i}}-\sqrt{\lambda_{j}}|\geq \frac{1}{C}|i-j|.
\end{eqnarray*}
So
\begin{equation}\label{eq10.07}
|i-j|\leq C|\langle k, \omega\rangle|.
\end{equation}

 \ \

\begin{itemize}
\item [(1.1.1)] $i\geq i_{0},\;j\geq j_{0}.$

\ \

By \eqref{eq10.001}, we have that, when $\omega\in \Pi_{m}$ such that
\eqref{eq10.04} holds true, the following inequality holds true:
\begin{eqnarray}\label{eq10.08}
|-\langle k, \omega\rangle +\rho i-\rho j|&=&|(-\langle k, \omega\rangle +\mu^{ml}_{kij})+(\rho i- \rho j-\mu^{ml}_{kij})|\nonumber\\
&\leq &\frac{|i-j|+1}{A_{k}}\gamma _{_{m}}+\frac{C_{1}(M)}{i}+
\frac{C_{2}(M)}{j}\nonumber\\
&\leq & \frac{|i-j|+1}{A_{k}}\gamma_{_{m}}+\frac{C_{1}(M)}{i_{0}}+
\frac{C_{2}(M)}{j_{0}},
\end{eqnarray}
where $C_{1}(M)>0$ and $C_{2}(M)>0$ are constants.

Thus
\begin{equation}\label{eq10.09}
Q^{(m)}_{kijl}\subset \left\{\omega\in\Pi_{m}\big||-\langle k, \omega\rangle +\widetilde{l}|<\frac{|\widetilde{l}|+1}{A_{k}}\gamma_{m}
+\frac{C_{1}(M)}{i_{0}}+\frac{C_{2}(M)}{j_{0}}\right\}\triangleq \widetilde{Q}_{k\widetilde{l}}.
\end{equation}
By \eqref{eq10.07}, one has
\begin{equation}\label{eq10.010}
|\widetilde{l}|\leq C |\langle k, \omega\rangle|\leq C |k|.
\end{equation}
Note that $k\neq 0.$ Then
$$\frac{d(-\langle k, \omega\rangle +\rho\widetilde{l})}{d\omega}> \frac{1}{2}|k|\geq \frac{1}{2}.$$
It follows that
\begin{equation}\label{eq10.011}
mes \widetilde{Q}_{k\widetilde{l}}\leq 4\left(\frac{|\widetilde{l}|+1}{A_{k}}\gamma_{m}+\frac{C_{1}(M)}{i_{0}}
+\frac{C_{2}(M)}{j_{0}}\right).
\end{equation}
Take
\begin{equation}\label{eq10.012}
j_{0}=i_{0}=|k|^{n+2}\gamma_{m}^{-1/3}.
\end{equation}
Then
\begin{eqnarray*}
mes \bigcup_{1\leq \widetilde{l}\leq C|k|} \widetilde{Q}_{k\widetilde{l}}
&\leq & \frac{C|k|\gamma_{m}}{A_{k}}+C\sum_{1\leq |\widetilde{l}|\leq C|k|}
(\frac{C_{1}(M)}{i_{0}}+\frac{C_{2}(M)}{j_{0}})\\
&\leq & \frac{C|k|\gamma_{m}}{A_{k}}+\gamma_{m}^{1/3}\frac{C|k|}{|k|^{n+2}}\\
&\leq &\frac{C\gamma_{m}^{1/3}}{|k|^{n+1}}.
\end{eqnarray*}
It follows from \eqref{eq10.09} that
\begin{equation}\label{eq10.013}
mes \bigcup_{\begin{array}{c}
                                i\geq i_{0} \\
                                j\geq j_{0} \\
                                |i-j|\leq C|k|
                              \end{array}}
                              \!\!\!Q_{kijl}^{(m)}\leq\frac{C\gamma_{m}^{1/3}}{|k|^{n+1}}.
 \end{equation}

  \item [(1.1.2)]
$i\leq i_{0} \;\;\mbox{or}\;\;j\leq j_{0}.$

\ \

By \eqref{eq10.07}, one has
$|i-j|\leq C |k|.$
In addition, $1\otimes \Lambda_{i}^{(m)}-(\Lambda_{j}^{(m)})^{T}\otimes 1$ is obviously Hermitian. Then
by the variation of eigenvalues for Hermitian matrix, we have
$$\Big|\frac{d}{d\omega}(-\langle k, \omega\rangle  +\mu^{ml}_{ijk})\Big|
\geq |k|-\Big|\frac{d \mu^{ml}_{ijk}}{d\omega}\Big|\geq \frac{1}{2}.$$
Therefore,
\begin{eqnarray}\label{eq10.015}
& & mes \!\! \!\!\!\bigcup_{\begin{array}{c}
                               1\leq i\leq i_{0} \\
                                |i-j|\leq C|k|
                              \end{array}}
                             \!\! \!\!\!Q_{kijl}^{(m)}
\leq
\sum_{\begin{array}{c}
 1\leq i\leq i_{0} \nonumber\\
|i-j|\leq C|k|
                              \end{array}}\!\!\!\frac{4(|i-j|+1)\gamma_{m}}{A_{k}}\leq\frac{C|k|\gamma_{m}i_{0}}{A_{k}}\\
                              &&
                              \leq C|k|^{n+3}\gamma_{m}^{2/3}\frac{1}{A_{k}}
\leq \frac{C\gamma_{m}^{2/3}}{|k|^{n+1}}.
\end{eqnarray}
Similarly, one has
\begin{eqnarray}\label{eq10.017}
mes \!\! \!\!\!\bigcup_{\begin{array}{c}
                               1\leq j\leq j_{0} \\
                                |i-j|\leq C|k|
                              \end{array}}
                             \!\! \!\!\!Q_{kijl}^{(m)}
                              \leq \frac{C\gamma_{m}^{2/3}}{|k|^{n+1}}.
\end{eqnarray}
\end{itemize}
 \item  [(1.2)] $k=0.$ \\
By \eqref{eq10.06} and \eqref{eq10.18}, one has $ Q_{kijl}^{(m)}=\varnothing,$ then
\begin{equation}\label{eq10.21}
mes Q_{kijl}^{(m)}=0.
\end{equation}
\end{itemize}
{\bf Case 2.}
$i=j$, one has $k\neq 0.$

\ \

 At this time, by Lemma \ref{lem10.2}, $$$$
\begin{equation}
\label{eq10.02}
-\langle k, \omega\rangle+\mu^{ml}_{kij}=-\langle k, \omega\rangle+O(\frac{\varepsilon_{0}}{|i|}).
\end{equation}
\begin{itemize}
  \item [(2.1)] Suppose $ \big|\langle k, \omega\rangle
\big|\geq\frac{2\gamma^{2/3}_{m}}{A_{k}}.$

\begin{itemize}
  \item [(2.1.1)]$i>\frac{C\varepsilon_{0}A_{k}}{\gamma_{m}^{2/3}}.$

\ \

By \eqref{eq10.02}, one has $$|-\langle k, \omega\rangle+\mu^{ml}_{kij}|\geq \frac{2\gamma^{2/3}_{m}}{A_{k}}-\frac{C\varepsilon_{0}}{i}>\frac{\gamma^{2/3}_{m}}{A_{k}}.$$
It follows from \eqref{eq10.04} that
$Q_{kiil}^{(m)}=\varnothing.$ Then
\begin{equation}\label{eq10.03}
mes Q_{kiil}^{(m)}=0.
\end{equation}
  \item [(2.1.2)] $i\leq \frac{C\varepsilon_{0}A_{k}}{\gamma_{m}^{2/3}}\triangleq \widetilde{k}.$
\ \

Note that $$\frac{d(-\langle k, \omega\rangle+\mu^{ml}_{kij})}{d\omega}=|k|+O(\frac{\varepsilon_{0}}{|i|})\geq \frac{1}{2}.$$ Then
\begin{equation}\label{eq10.20}
mes \bigcup_{i\leq \widetilde{k}}Q^{(m)}_{kiil}\leq \frac{4\widetilde{k}\gamma_{m}}{A_{k}}\leq C\gamma_{m}^{1/3}.
\end{equation}
\end{itemize}
 \item [(2.2)] Suppose $|\langle k, \omega\rangle |<\frac{2\gamma^{2/3}_{m}}{A_{k}}.$

Let
$$\widetilde{Q}_{k}=\left\{\omega\in\Pi_{m}\big||\langle k, \omega\rangle |<\frac{2\gamma^{2/3}_{m}}{A_{k}}\right\}.$$
Note that $|\frac{d(\langle k, \omega\rangle)}{d\omega}|=|k|\geq 1.$ Then
$$mes \widetilde{Q}_{k} \leq \frac{4\gamma^{2/3}_{m}}{A_{k}},$$
and
\begin{equation}\label{eq10.19}
 mes \!\! \bigcup_{k\in \mathbb{Z}^{n}\setminus\{0\}}
                             \!\!\widetilde{Q}_{k} \leq \sum_{k\in \mathbb{Z}^{n}\setminus\{0\}}\frac{C\gamma^{2/3}_{m}}{A_{k}}\leq C\gamma^{1/3}_{m}.
\end{equation}
\end{itemize}
\ \

Combining \eqref{eq10.013}, \eqref{eq10.015}, \eqref{eq10.017},\eqref{eq10.21}, \eqref{eq10.03}, \eqref{eq10.20} and \eqref{eq10.19}, we have
\begin{eqnarray}\label{eq10.016}
mes \bigcup_{|k|\leq K_{m}}\bigcup_{i=1}^{\infty}\bigcup_{j=1}^{\infty}\bigcup_{l=1}^{4}Q^{(m)}_{kijl}
\leq C\gamma_{m}^{1/3}.
 \end{eqnarray}
 Let
 $$\Pi^{+-}_{m+1}=\Pi_{m}\backslash
 \bigcup_{|k|\leq K_{m}}\bigcup_{i,j=1}^{\infty}\bigcup_{l=1}^{4}Q^{(m)}_{kijl}.$$
 Then we have proved the following Lemma \ref{lem10.1}.

\begin{lemma}\label{lem10.1}
$$mes\Pi^{+-}_{m+1}\geq mes \Pi_{m}-C\gamma_{m}^{1/3}.$$
\end{lemma}
\section{Proof of Theorems}
Theorem 2.1 is a more exact statement of Theorem 1.1.
Let
$$\Pi_{\infty}=\bigcap_{m=1}^{\infty} \Pi_{m},$$
and $$\Psi_{\infty}=\lim_{m\rightarrow \infty}\Psi_{0}\circ\Psi_{1}\circ \cdots \circ \Psi_{m}.$$
By \eqref{eq33} and \eqref{eq33H}, one has
$$\Psi_{\infty}: \mathbb{T}^{n}\times \Pi_{\infty}\rightarrow \mathbb{T}^{n}\times \Pi_{\infty},$$
$$||\Psi_{\infty}-id||\leq \varepsilon^{1/2},$$
and, by \eqref{eq34},
$$H_{\infty}=H \circ \Psi_{\infty}=\sum_{j=0}^{\infty}\langle \Lambda_{j}^{\infty}u_{j}, \overline{u}_{j}\rangle,$$
where
$\Lambda_{j}^{\infty}=\Lambda_{j}^{(0)}+Q_{j}^{(0)},$ and
$Q_{j}^{(0)}$ is independent of time, $Q_{0}\in \mathbb{R},$ $Q_{j}\in gl(\mathbb{R}, 2)$ with $j\neq 0.$

This completes the proof of Theorem \ref{thm2.1}.

\section*{Acknowledgements}

The work was supported in part by the National
Nature Science Foundation of China (Grants Nos. 11601277 and 11771253) and by the Shandong Provincial Natural Science
Foundation, China (Grant ZR2016AM20).


\begin{thebibliography}{00}
\bibitem{Arnold}V. I. Arnold, Mathematical methods of classical mechanics,
Translated from the 1974 Russian original by K. Vogtmann and A. Weinstein. Corrected reprint of the second (1989) edition.
Graduate Texts in Mathematics, 60 (1991) 229-234.

\bibitem{Bambusi17arxiv}D. Bambusi, B. Gr\'{e}bert, A. Maspero, D. Robert, Reducibility of the quantum harmonic oscillator in d-dimensions with polynomial time dependent perturbation, arXiv:1702.05274, 2017.

\bibitem{BG} D. Bambusi, S. Graffi, Time Quasi-Periodic Unbounded
Perturbations of Schr\"{o}dinger Operators and KAM Methods. Commun.
Math. Phys. 219 (2001) 465-480.

\bibitem{Bambusi17CMP}D. Bambusi, Reducibility of 1-d Schr\"{o}dinger equation with time quasiperiodic unbounded perturbations, II, Commun. Math. Phys. 353 (2017) 353-378.


\bibitem{Baldi} P.Baldi, M.Berti, R. Montalto, KAM for quasi-linear and fully nonlinear
forced perturbations of Airy equation. Math. Ann. 359(1每2) (2014) 471每536.



\bibitem{Baldi2016} P. Baldi, M. Berti, R. Montalto, KAM for autonomous quasi-linear perturbations of KdV. Ann. I. H.
Poincar\'{e} (C) Anal. Non Lin\'{e}aire 33 (2016) 1589-1638.

\bibitem{Baldi16mKdv} P. Baldi, M. Berti, R. Montalto, KAM for autonomous quasi-linear perturbations of mKdV. Bollettino
Unione Matematica Italiana, 9 (2016) 143-188.

\bibitem{Baldi17}P. Baldi, M. Berti, E. Haus, R. Montalto, Time quasi-periodic gravity water waves in finite depth.
Preprint arXiv:1602.02411, 2017.

\bibitem{Massimiliano-Berti} M. Berti, P. Bolle, Sobolev quasi periodic solutions of multidimensional
wave equations with a multiplicative potential, Nonlinearity, 25 (2012) 2579-2613.

\bibitem{Berti2017} M. Berti, R. Montalto, Quasi-periodic standing wave solutions of gravity-capillary water waves, to appear
on Memoirs of the American Math. Society MEMO 891, 2017.


\bibitem{Bogoljubov1969}N. N. Bogoljubov, Ju. A. Mitropoliskii, A. M. Samo$\breve{{\i}}$lenko, Methods of accelerated convergence in nonlinear mechanics, Springer-Verlag, New York, 1976, translated from the original Russian, Naukova Dumka, Kiev, 1969.

\bibitem{Bourgain-Goldstin2000}J. Bourgain, M. Goldstein, On nonperturbative localization
with quasi-periodic potential, Ann. of Math. 152 (2000) 835-879.


\bibitem{Chierchia} L. Chierchia, D. Qian, Moser's theorem for lower dimensional tori, J. Diff. Eqs. 206 (2004) 55-93.

\bibitem{Chierchia-You}L. Chierchia, J. You, KAM tori for 1D nonlinear wave equations with periodic boundary conditions, Comm. Math. Phys. 211 (2000) 497每525.

\bibitem{Combescure1987Ann}M. Combescure, The quantum stability problem for time-periodic perturbations of the harmonic oscillator, Ann. Inst. Henri Poincar\'{e},
47(1) (1987) 63-83.

\bibitem{Duclos-Stovicek96}P. Duclos, P. \v{S}\v{t}ov\'{\i}\v{c}ek, Floquet Hamiltonians with pure point spectrum, Comm.
Math. Phys. 177(2) (1996) 327每347.

\bibitem{Duclos-Stovicek02}P. Duclos, P. \v{S}\v{t}ov\'{\i}\v{c}ek, M. Vittot, Perturbation of an eigen-value from a dense point spectrum: A general Floquet Hamiltonian, Ann. Inst. Henri Poincar\'e, 71(3) (1999) 241-301.

\bibitem{Eliasson92}
L. H. Eliasson, Floquet solutions for the 1-dimensional quasi-periodic Schr\"{o}dinger equation, Commun. Math. Phys., 146(3) (1992) 447-482.

\bibitem{Eliasson98}
L. H. Eliasson, Reducibility and point spectrum for linear quasi-periodic skew-products, Proceedings of the International Congress of Mathematicians, 2 (1998) 779-787.

\bibitem{eliasson-kuksin}
L. H. Eliasson, S. B. Kuksin, On reducibility of Schr\"odinger equations with quasiperiodic in time potentials, Comm.
Math. Phys. 286(1) (2009) 125-135.

\bibitem{Feola2016}
R. Feola, KAM for quasi-linear forced hamiltonian NLS, preprint arXiv:1602.01341, 2016.

\bibitem{Feola2015} R. Feola, M. Procesi Quasi-periodic solutions for fully nonlinear forced reversible Schr\"{o}dinger equations,
J. Diff. Eqs., 259(7) (2015) 3389-3447.

\bibitem{Giuliani} F. Giuliani, Quasi-periodic solutions for quasi-linear generalized KdV equations. J. Diff. Eqs.
262 (2017) 5052-5132.

\bibitem{Jackson}D. Jackson, The theory of approximation, American Mathematical Soc., 1930.

\bibitem{Jorba92}A. Jorba, C. Sim\'{o}, On the reducibility of linear differential equaitons with quasiperiodic coefficients, J. Diff. Eqs. 98 (1992) 111-124.

\bibitem{Klein2005}S. Klein, Anderson localization for the discrete one-dimensional quasi-periodic Schr\"{o}dinger operator with potential defined by a Gevrey-class function, J. Funct. Anal. 218(2) (2005) 255-292.

\bibitem{Kuk1} S. B. Kuksin,  Nearly integrable infinite-dimensional Hamiltonian systems. (Lecture Notes in Math. 1556). Springer-Verlag, New York, 1993.

\bibitem{Li}
J. Li, Reducibility, Lyapunov exponent, pure point spectra property for quasi-periodic wave operator, arXiv preprint arXiv:1706.06713, 2017.

\bibitem{Li-Zhu}J. Li, C. Zhu, On the reducibility of a class of finitely differentiable quasi-periodic linear systems,
 J. Math. Anal. Appl., 413(1) (2014) 69-83.

\bibitem{Liang17JDDE}Z. Liang, X. Wang, On reducibility of 1d wave equation with quasiperiodic in time potentials, J Dyn. Diff. Equat. DOI:10.1007/s10884-017-9576-4.

\bibitem{Riccardo Montalto1} R. Montalto, Quasi-periodic solutions of forced Kirchhoff equation, Nonlinear Differ. Equ. Appl. 24(9) (2017)
https://doi.org/10.1007/s00030-017-0432-3.

\bibitem{Riccardo Montalto} R. Montalto, A reducibility result for a class of linear wave
equations with unbounded perturbations on $\mathbb{T}^{d}$, arXiv:1702.06880, 2017.


\bibitem{Poschel1996}J. P\"oschel, A KAM-theorem for some nonlinear partial differential equations,
 Ann. Scuola Norm. Sup. Pisa, Cl. Sci., IV Ser. 15,
     23 (1996) 119-148.

\bibitem{Salamon1989}D. Salamon, E. Zehnder, KAM theory in configuration space, Comment. Math. Helv. 64 (1989) 84-132.

\bibitem{Salamon2004}D. Salamon, The Kolmogorov-Arnold-Moser theorem, Mathematical Physics Electronic Journal, 10(3) (2004) 1-37.

\bibitem{Wang17nonlinearity}Z. Wang, Z. Liang, Reducibility of 1D quantum harmonic oscillator perturbed by a quasiperiodic potential with logarithmic decay, Nonlinearity, 30 (2017) 1405-1448.

\bibitem{Way} Wayne, C. E., Periodic and quasi-periodic solutions of nonlinear
    wave equations via KAM theory, Commun. Math. Phys., 127 (1990) 479-528.

\bibitem{Roberto}F. Roberto, P. Michela, Quasi-periodic solutions for fully nonlinear forced reversible Schr\"{o}dinger equations, J. Diff. Eqs. 259 (2015) 3389每3447

\bibitem{yuan-zhang}X. Yuan, K. Zhang, A reduction theorem for time dependent Schr\"{o}dinger operator with finite differentiable unbounded perturbation. J. Math. Phys. 54 (2013) 052701.

\bibitem{E. Zehnder} E. Zehnder,
Generalized implicit function theorems with applications to small divisor problems I II,
Comm. Pure Appl. Math. 28 (1975), 91-140; 29 (1976) 49-111.

\bibitem{1}https://www.youtube.com/watch?v=6KYOeJx4srU.

\bibitem{2}https://www.youtube.com/watch?v=6oTVQCF7Z2g.

\bibitem{3}https://www.youtube.com/watch?v=Fwwd-hoBkj8.
\end{thebibliography}
\end{document}